\documentclass[10pt]{amsart}
\usepackage{amsfonts}
\usepackage{amssymb,amsthm,color,tikz}
\usepackage{amsmath}
\usepackage{comment}

\usepackage{graphicx}
\usepackage{dsfont}
\usetikzlibrary{patterns}
\usepackage[a4paper, centering]{geometry}

\usepackage{tikz}
\usetikzlibrary{calc}

\usepackage[normalem]{ulem}
\usepackage{url}

\newcommand{\ds}{\displaystyle}

\newcommand{\R}{\mathbb{R}}
\newcommand{\N}{\mathbb{N}}

\newcommand{\E}{\mathbb{E}}
\newcommand{\Pa}{\mathbb{P}}

\newtheorem{theorem}{Theorem}
\newtheorem{lemma}[theorem]{Lemma}
\newtheorem{proposition}[theorem]{Proposition}
\newtheorem{corollary}[theorem]{Corollary}

\newtheorem{definition}[theorem]{Definition}
\newtheorem{remark}[theorem]{Remark}

\def\H{{\mathcal H}}
\def\Hd{{\mathcal H}^{d-1}}
\def\PP{{\mathcal P}}

\def \d {\mathrm{d}}

\newcommand{\EEE}{\color{black}}

\newcommand{\Om}{\Omega}
\newcommand{\lb}{\lambda}

\newcommand{\sm}{\setminus}

\newcommand{\sq}{\subseteq}
\newcommand{\ov}{\overline}
\newcommand{\vphi}{\varphi}
\newcommand{\vps}{\varepsilon}
\newcommand{\eps}{\varepsilon}

\newcommand{\om}{\omega}
\newcommand{\nif}{{n\to +\infty}}

\newcommand{\bp}{\begin{proof}}
\newcommand{\ep}{\end{proof}}

\newtheorem{prop}[theorem]{Proposition}

\begin{document}
\title[Shape optimisation of the maximum gradient of the torsion function]{
Shape optimisation of the maximum of the gradient of the torsion function on convex sets}
\author[]{Michiel van den Berg, Dorin Bucur, Micka\"{e}l Nahon}

\address[Michiel van den Berg]{
School of Mathematics, University of Bristol\\
Fry Building, Woodland Road\\
Bristol BS8 1UG, United Kingdom
}
\email{mamvdb@bristol.ac.uk}

\address[Dorin Bucur]
{Laboratoire de Math\'ematiques, Universit\'e Savoie Mont Blanc \\
UMR CNRS  5127\\
Campus Scientifique,
73376 Le-Bourget-Du-Lac, France
}
\email{dorin.bucur@univ-savoie.fr}

\address[Micka\"{e}l Nahon]
{
Universit\'e Grenobles Alpes, CNRS\\
Laboratoire Jean Kunzmann\\
38000 Gr\'enoble, France.
}
\email{mickael.nahon@univ-grenoble-alpes.fr}
\date{23 October 2026}
\maketitle

\begin{abstract}
We analyse the boundary behaviour of the gradient of the torsion function $w_\Om$ in convex domains $\Om\subset\mathbb{R}^d$. We prove that its $L^\infty$-norm is attained at a boundary point and establish the stability of this norm under general (convex) perturbations. For problems of the form
 $$\sup\{\|\nabla w_{\Om}\|_\infty: \Om \subset \R^d,\textup{open, bounded and convex}, G(\Om)=m\},$$ where  $G$ is either volume or perimeter,  we show the existence of maximisers which are $C^1$-smooth, rotationally symmetric about an  axis and contain a $(d-1)$-dimensional ball  (of  estimated size) in their boundary. Other constraints, whether geometric (such as diameter or circumradius) or variational (such as torsion or first Dirichlet eigenvalue), are also briefly discussed, and some of their characteristic behaviours are highlighted. \end{abstract}
\vskip 1truecm \noindent   {Mathematics Subject
Classification ({2020})}: { 28A75, 35B50, 35K05, 60J65}.

 \section{Introduction}

Let $\Om \subset \R^d$ be a bounded, open, convex set. The torsion function associated to $\Om$, denoted by $w_\Om$, is the  (weak) solution of
\begin{equation}\label{da01}
\begin{cases}
-\Delta w_\Om = 1 & \text{in } \Omega,\\[2mm]
w_\Om = 0 & \text{on } \partial\Omega.
\end{cases}
\end{equation}
While   $w_\Om \in C^\infty (\Om)\cap C^0(\ov \Om)$, $\nabla w_\Om$ may  not extend continuously to the boundary. The  supremum  of its length, denoted $\|\nabla w_\Om \|_\infty$,  is finite but  not attained in the interior of $\Om$.

 In Saint-Venant's theory of torsion, the torsional rigidity $\int_\Om w_\Om$ of a two-dimen\-sional domain $\Om$ quantifies the resistance of a cylindrical rod with cross section
$\Om$  to twisting under an applied torque. Of particular interest are the {\it fail} or {\it maximal stress points}, namely the points where the maximum of the gradient of the torsion function is attained and where the elastic material would change its behaviour towards plasticity, under a critical  torque value. The location of these points on the boundary of a given domain led to a  celebrated conjecture by Saint-Venant in his treaty from 1856 (see \cite{SaintVenant1856} and  \cite{Sw89}). This conjecture received a lot attention until Sweers disproved it in 1989  \cite{Sw89} (see as well \cite{Ra90}). To date there are no analytical results identifying the location of the fail points in general domains; most results are numerical.

If the convex domain has a smooth boundary, then the maximum of the gradient of the torsion function is attained at a point of the boundary. When the boundary is not smooth, however, the gradient of the torsion function need not extend continuously to the boundary. Thus, establishing that the maximum of the gradient is actually attained at a boundary point requires some care. In particular, one must first verify that the gradient is well defined at every boundary point before addressing the existence question of fail  points on a non-smooth boundary.

A series of recent papers \cite{LY, LXYY26, HS} deal with the analysis of the fail points and on the maximum stress. The recent landmark paper \cite{Burd} of Burdzy, Ftouhi and Mariano was a main motivation for the current paper. Their paper deals with the shape optimisation problem of two-dimensional convex domains for the maximum stress $\|\nabla w_\Om\|_\infty$ under an area or a perimeter constraint on $\Om$. They establish, for both constraints, the existence of optimal domains, the fact that they have a $C^1$-smooth boundary and the boundary should contain a segment of strictly positive length, where the fail point is located (see also the paper by Huang \cite{Hu25}). Their analysis is carried out within a probabilistic framework, based on the interpretation of $w_
\Om (x)$ as the expected exit time from $\Om$ of a Brownian motion starting at $x$.

Our paper aims to complement the results of \cite{Burd}, and to provide a new perspective on the problem. The main results are Theorems  \ref{th_existence_failpoint}, \ref{th_stability_maxstress},  \ref{th_inward_max}  and \ref{d20} below. Theorem \ref{th_existence_failpoint} establishes the existence of fail points for every open, bounded, convex domain in arbitrary dimension, independently of the regularity of its boundary. A key technical step is to prove that $w_\Om$ is differentiable up to the boundary. We provide a short, self-contained proof of this fact, while referring to \cite{LW06} for a  general treatment of the issue. Boundary differentiability, in turn, yields a natural extension of $\nabla w_\Om$ to $\Om\cup\partial\Om$. Moreover, the modulus of the extended gradient, $|\nabla w_\Om|$, is upper semicontinuous on $\Om\cup\partial\Om$, and the maximum stress $\Vert\nabla w_\Om\Vert_\infty$ is attained on $\partial\Om$.

For any $x\in\partial\Om$ we denote by  $\nu_x$  an outer normal vector of $\Om$ at $x$. If $\nu_x$ is not unique (e.g. the point $x$ is angular) then, in the statement below, $\nu_x$ is chosen  such that for $t>0$ small, $x-t \nu_x\in \Om$. 
\begin{theorem}[The gradient of the torsion function at boundary points]\label{th_existence_failpoint}
Let $\Om \subset \R^{d }$ be a bounded, open and convex set.   Then the following assertions hold true.
\begin{enumerate}
\item [1.] For any $\ov{x}\in\partial\Om$, there exists a vector of $\R^d$, denoted by $\nabla w_\Om(\ov{x})$, such that
\begin{equation}\label{eq_differentiable_bord}
w_\Om(x)=\left(\nabla w_\Om(\ov{x}),x-\ov{x} \right)+o_{x(\in\Om)\to \ov{x}}\left(|x-\ov{x}|\right),
\end{equation}
and $w_\Om$ restricted to $\ov{\Om}$ is differentiable.
In particular, $\nabla w_\Om(\ov{x})=0$ if $\ov{x}$ is an angular point.
\item[2.] For any $\ov{x}\in\partial\Om$, we have.
$$|\nabla w_\Om(\ov{x})|=\lim_{t\to 0^+}-\frac{\partial w_\Om}{\partial \nu_{\ov{x}}}(\ov{x}-t\nu_{\ov{x}})=\limsup_{x(\in\Om)\to \ov{x}}|\nabla w_\Om(x)|.$$
\item [3.]   The mapping $x(\in \partial \Om) \to |\nabla w_\Om (x)|$ is upper semicontinuous.
\item [4.]  $\sup_{x\in\Om}|\nabla w_\Om(x)|=\sup_{x\in\ov{\Om}}|\nabla w_\Om(x)|$, and the supremum is attained at some point of $\partial\Om$.
\end{enumerate}
\end{theorem}
In Theorem \ref{th_stability_maxstress} below, we prove the stability of the maximum stress in the class of convex sets. The following example may give an idea on the difficulty:  consider the sequence of regular $n$-gons inscribed in the unit circle. This sequence converges to the unit disk in the Hausdorff metric, while the corresponding torsion functions (extended by $0$ outside their domain) converge strongly in $H^1$. On the boundary of the disk, the gradient of the torsion function has constant magnitude. In contrast, at every vertex of the $n$-gon, the corresponding gradient vanishes, causing the trace of the gradient on the boundary of the $n$-gon to exhibit increasingly strong oscillations. The stability result of Theorem \ref{th_stability_maxstress} shows, nevertheless, that the maximum of the norm remains continuous under this convergence.

\begin{theorem}\label{th_stability_maxstress}
Let $(\Om_n)_{n\in\N},  \Om \subset \R^d$ be bounded, open and convex. If $\Om_n ^c \to \Om^ c$ in the Hausdorff metric then
$$\|\nabla w_{\Om_n}\|_\infty \to \|\nabla w_{\Om}\|_\infty. $$
\end{theorem}

 In Theorem \ref{th_inward_max} below we extend the results of \cite{Burd} concerning the existence of a line segment in the boundary of maximisers of the maximum stress under volume and perimeter constraints. We show that this phenomenon is a particular instance of a more general mechanism, which applies in arbitrary dimension to \textit{inward-maximisers} of the maximum stress.

\begin{theorem}\label{th_inward_max}
Let $\Om\subset\R^d$ be a bounded, open, convex set. Assume that for any open convex subset $\Om'\subset\Om$ we have
$$\Vert\nabla w_\Om\Vert_\infty\geq\Vert \nabla w_{\Om'}\Vert_{\infty}.$$
Let $x\in\partial\Om$ be a fail point of $\Om$, and let $\nu_x$ be an outward normal vector of $\Om$ at $x$. Then
$$\left\{y\in\R^d:(y-x,\nu_x)=0\text{ and }|y-x|\leq \pi^{-1}\Vert \nabla w_\Om\Vert_{\infty}\right\}\subset\partial\Om,$$
and the flat ball defined above does not contain any other fail point.
\end{theorem}
As we shall see in Section \ref{sec_inwardmax}, this theorem applies to any maximiser of the maximum stress under a constraint satisfying suitable monotonicity properties. This framework includes constraints on volume and perimeter, as well as on the diameter, inradius, circumradius, torsional rigidity, and the first Dirichlet eigenvalue of the Laplacian.

The following free boundary problem  provides some intuition on the presence of flat parts in the boundaries of inward-maximal sets. All its solutions satisfy  the hypothesis of Theorem \ref{th_inward_max}. Given a bounded open set $D$, consider
\[
\max \{\|\nabla w_\Omega\|_\infty:\ \Omega\subset D,\ \Omega\ \mbox{convex}\}.
\]
  If $D$ is also convex, then it is readily observed that the solution to this problem is a chord domain, namely, the intersection of $D$ with a halfspace. Thus, $\Omega$ is one of the two subdomains obtained by splitting $D$ with a hyperplane. The maximum of the gradient of the corresponding solution is attained at a point on this hyperplane.

Finally we discuss the maximisation problem for the maximum stress under a volume or a perimeter constraint in any dimension $d\ge 2$. We prove the existence of solutions which are  rotationally  symmetric about an axis. These results complement the ones of  \cite{Burd}. For the perimeter constraint we further prove that all maximisers are of this type and, in two dimensions, that the class of maximisers is closed under convex Minkowski addition (which does not exclude the uniqueness of the maximiser). Below, let $G(\Om)$ denote either the volume or the perimeter of $\Om$.

 \begin{theorem}\label{d20}

 Let $m>0$. The following problem
 \begin{equation}\label{d43}
c:= \max\{\|\nabla w_\Om\|_\infty : \Om \subset \R^d, \Om \mbox{ open and convex }, G(\Om)=m\}
 \end{equation}
 has at least one solution.

 Assume $\Om \subset \R^{d-1} \times \R_+$ is solution of problem \eqref{d43} such that $0 \in \partial \Om$ is a fail point.
Then
 \begin{enumerate}
 \item [1.] The set $\Om$ is an inward maximiser for the maximum stress of the torsion function. In particular,
 $$B^{d-1}_{\pi^{-1}\Vert \nabla w_\Om\Vert_{\infty}}\times\{0\}\subset\partial\Om.$$

 \item [2.]  The set $\Om$ is of class $C^1$.
  \item [3.] The   set $\Om^*$  obtained by axial Steiner symmetrisation of $\Om$ around $\R\ov e_d$ is also a solution of problem \eqref{d43}.
 \item [4.]
If $G=\Hd$, then we have the following.
 \begin{itemize}
 \item [(i)]
  The sets $\Om$ and $\Om^*$  coincide.  If $d \ge 3$ then $\partial \Om^*$ does not contain  more than two flat parts in the boundary. If two flat parts are present, then they are balls of dimension $d-1$ in  hyperplanes orthogonal to $\ov e_d$.
\item [(ii)] In $\R^2$,   if $\Om_0$ and $\Om_1$ are two solutions as above, then for every $t \in (0,1)$ their Minkowski sum
 $(1-t) \Om_0+t\Om_1$ is also a solution.
 \end{itemize}
 \end{enumerate}
  \end{theorem}

In order to provide a broader perspective, we also consider other constraints, such as the diameter, circumradius, torsional rigidity, and first eigenvalue, and we briefly discuss some qualitative properties of the corresponding maximisers.\bigbreak

\textbf{Organisation of the paper.} In Section \ref{sec_preli} we collect several useful lemmas on the boundary behaviour of the gradient. In Section \ref{sec_gradient}, we establish the boundary differentiability of the torsion function and show that the maximum stress is reached at a fail point of the boundary, thereby proving Theorem \ref{th_existence_failpoint}. In Section \ref{sec_stability}, we establish the stability of the maximum stress for   a  sequence of convex sets (Theorem \ref{th_stability_maxstress}). In Section \ref{sec_inwardmax}, we prove that the fail points of inward  maximisers   lie on a flat part of the boundary. In Section \ref{sec_areaperimeterconstraint}, we study maximisers of the maximum stress under volume or perimeter constraints, establishing Theorem \ref{d20}. Finally, in Section \ref{sec_furtherremarks}, we study the maximisation of the maximum stress under other geometrical or variational constraints (diameter, circumradius, torsional rigidity and first Dirichlet eigenvalue).

\section{Preliminaries}\label{sec_preli}
Let $\Om \subset \R^d$ be a bounded, open, convex set. The torsion function associated to $\Om$ is the  weak solution of \eqref{da01},
defined as the unique minimiser of the torsion energy 
 $$u \in H^1_0(\Om)\mapsto  \frac 12 \int_\Om |\nabla u|^2 \d x -\int_\Om u \d x.$$
The minimiser, $w_\Om$, satisfies the Euler-Lagrange equation
 $$w_\Om\in H^1_0(\Om), \; \forall \vphi \in H^1_0(\Om) \;\; \int_\Om \nabla w_\Om \nabla \vphi \d x = \int_\Om \vphi \d x.$$
 The torsional rigidity of $\Om$ is $\mathcal P(\Om):= \int_\Om w_\Om =\int_{\Om}|\nabla w_\Om|^2$.

 By convention, the torsion function $w_\Om \in H^1_0(\Om)$ is seen as a function on $\R^d$ extended by $0$ on $\R^d \sm \Om$. If $\Om= \emptyset$, then $H^1_0(\Om)=\{0\}$ so $w_\Om=0$.

It was shown in \cite{vdBC} that if $\Om\subset \R^d$ is a non-empty open set, then $w_{\Om}$ is bounded if and only if  $\lb (\Om)$, the bottom of the spectrum of the Dirichlet Laplacian, is strictly positive. Moreover for a dimensional constant $c_d$ we have that
$$1\le \|w_\Om\|_\infty \lb(\Om) \le c_d.$$ The torsion function $w_\Om$ is formally the pointwise limit of the increasing sequence  $R \to w_{\Om \cap B_R}$ for $R \to +\infty$. The limit function $w_\Om$ defined in this way is independent on the choice of origin of the centre of the balls. Moreover, it satisfies $-\Delta w_\Om =1$ in the sense of distributions on $\Om$.
We recall that for bounded convex sets,
\begin{equation}\label{d00}\frac 2d  \|w_\Om \|_\infty\le  \|\nabla w_\Om\|_\infty ^2 \le 2 \|w_\Om \|_\infty.
\end{equation}
Both inequalities are sharp, the first one being due to  Weinberger \cite{WH} with equality on a ball  and the second is due to  Payne \cite[Theorem 3]{P2} with equality on $\R^{d-1}\times (0,h),\,h>0$. See also Sperb \cite[(6.12)]{RS}.
Let us point out that the second inequality holds as well for unbounded convex sets for which $ \lb(\Om) >0$ (and so $\|w_\Om\|_\infty <+\infty$). In fact, for such a set $\Om$, the sequence $(w_{\Om \cap B_R})_R$ is locally uniformly bounded in $W^{1, \infty}$, meaning that the second inequality passes to the limit when $R\to \infty$.

 We recall that $x_0\in \partial \Om$ is called a {\it corner point} (or angular point) if the normal  cone to  $\Om$ at $x$
 $$\mathcal N_\Om(x) = \{ v \in \R^d : (v,y-x) \le 0, \; \;  \forall  y\in \Om\}$$
  is not a singleton. In other words, there exist two different hyperplanes of normals $ \nu_1\not=  \nu_2$, passing through $x_0$ and leaving $\Om$ on one side.
  In such an angular point, we have that
\begin{equation}\label{eq_decay_angle}
w_\Om(x)\le C_\Om |x-x_0|^\alpha,
\end{equation}
where
$\alpha= \min \{2, \frac\pi\om\}$
and  $\omega$ denotes the angle in the plane  $span\{\nu_1, \nu_2\}$ \cite[Chapter 4]{Gr85}. See also \cite{JCA2003} for applications in shape optimisation.

We call a set $\Om$ a chord domain if it is obtained by intersecting $\Om$ with a half-space $\mathcal H^+$, where  $\mathcal H$ denotes a hyperplane.

An open ball centred at $x$ of radius $r$ is denoted $B_r(x)$. If $x=0$, we simply write $B_r$. Throughout the paper,  we denote by $\omega_d$ the volume of the unit ball in $\R^d$, by $|\Om|$ the volume of $\Om$ and by
$\Hd(\partial\Om)$ its perimeter.

The Hausdorff distance between two closed sets $F_1,F_2$ is given by
 $$d_H(F_1, F_2)= \inf \{\vps >0 : F_1\sq F_2\oplus B_\vps, F_2\sq F_1\oplus B_\vps\},$$
 where $\oplus$ means Minkowski sum.
 A sequence $(\Om_n)_n$ is said to converge in the Hausdorff complementary topology (denoted $H^c$) to $\Om$, if $d_H(\R^d \sm \Om_n, \R^d\sm \Om) \to 0$. If $\Om_n, \Om$ are open, uniformly bounded,  convex sets such that $\Om_n \to \Om$ in $H^c$, then
 $w_{\Om_n} \to w_\Om$ strongly in $H^1(\R^d)$. This last property is called $\gamma$-convergence. We refer to \cite{BB05} for an exhaustive analysis of the geometric stability of the torsion function.

 If $\Om\subset \R^2$ is a convex polygon, then $w_\Om \in H^{2+\delta}(\Om)$, for some $\delta >0$ depending on angles and relative length of the edges (see  \cite[Chapter 4]{Gr85} and \cite{BDLN92}).

 In a coordinate system, a point $x=(x_1, \dots, x_d)$ may be denoted by $x=(x', x_d)$ or $(\tilde x, x_d)$, where $x'= (x_1, \dots, x_{d-1})$.

We denote
$$\mathcal K_{\R^d}= \{\Om \subset \R^d : \Om \mbox{ is open, bounded and convex}, \Om \not=\emptyset\}.$$

\section{Boundary behaviour of the gradient}\label{sec_gradient}
 In this section we provide a description of the behaviour of the gradient of the torsion function at boundary points of a (not necessarily smooth) convex set, with particular emphasis on its precise definition and its continuity or discontinuity properties. The existence of the gradient at boundary points is discussed in detail in \cite{LW06}. For completeness, we provide below a self-contained definition, together with a characterisation of the slope of the gradient at a boundary point. In Theorem \ref{th_existence_failpoint}, we further establish the upper semicontinuity of the gradient magnitude  up to the boundary and prove the existence of a boundary point at which its maximum is attained.

\smallskip
\noindent{\bf Behaviour of the gradient near corner points.}
The first result describes the pointwise behaviour of the gradients of the torsion functions associated with a sequence of chord domains near an angular boundary point.
\begin{lemma}\label{d04}
Let $\Om \in \mathcal K_{\R^d}$ and assume $\ov{x}\in \partial \Om$ is an angular point.
Assume $\Om_n=\Om \cap \H_n^+$ are chord domains and  $x_n \in \Om \cap \partial \Om_n$ is such that $x_n \to \ov{x}$. Then
$$\lim_{\nif} |\nabla w_{\Om_n}(x_n) |=0.$$

\end{lemma}
Note that we require only convergence of the points, with no assumption on the convergence of the sequence of sets $(\Om_n)_n$. This result provides a slight refinement of the statement below and will be useful in what follows.
\begin{corollary}\label{cor_gradientcornerpoint}
Let $\Om \in \mathcal K_{\R^d}$ and  assume $\ov{x}\in \partial \Om$ is a corner point. If $x_n\in \Om$, $x_n \to \ov{x}$, then
$$\lim_{\nif} |\nabla w_{\Om}(x_n) |=0.$$
\end{corollary}
\begin{proof}
For every $x_n \in \Om$, we take for $\Om_n$ the chord domain  obtained by   partitioning $\Om$ with the tangent hyperplane to the level set of $w_\Om$ which passes through $x_n$ and which contains $\{x \in \Om : w_\Om (x) > w_\Om(x_n)\}$. Recall that the level sets of $w_\Om$ are convex. A barrier argument gives that  $|\nabla w_{\Om_n}(x_n) |\ge |\nabla w_{\Om}(x_n) |$, and Lemma \ref{d04} finishes the proof.
\end{proof}
 \begin{proof}[Proof of Lemma \ref{d04}]
 We recall from Gilbarg-Trudinger \cite{GT01}, Section 3.4, (3.15) the following classical estimates for the gradient of the solution of the equation  $-\Delta u =f $ in $H^1(Q)$.

\smallskip

 If $Q=(-r,r)^ d$, then
 $$|\nabla u(x)| \le \frac dr \sup _{\partial Q} |u|+ \frac r2 \sup _Q|f|.$$

\smallskip

If $Q=B_r$, then for all $x \in B_{\frac r2},$
 \begin{equation}\label{d12}
 |\nabla u(x)| \le C\Big(\frac 1r \sup _{Q} |u| +r\sup _Q|f|\Big),
 \end{equation}
 where $C$ is a dimensional constant.

\smallskip

If $Q=B_r\cap \{x_d >0\}$ and $u=0$ on $\{x_d =0\}$, then
 for all $x \in B_{\frac r2}\cap \{x_d \ge 0\},$
  \begin{equation}\label{d13}
 |\nabla u(x)| \le C\Big(\frac 1r \sup _{Q} |u| +r\sup _Q|f|\Big).
  \end{equation}

  \smallskip

 Coming back to the proof of Lemma \ref{d04}, we can assume $\ov{x}$ is the origin of a system of coordinates.
 We denote, for any $x\in\R^d$,
 $$\vps(x)=\frac{w_\Om(x)}{|x|}$$
 and $\lim_{x\to 0}\vps(x)=0$ by the estimate \eqref{eq_decay_angle}.

 Let us denote $d_n = d(x_n, \partial \Om)$ and $r_n=|x_n|$ (see Figure \ref{d14f}).

 \begin{figure}
\begin{center}
  \includegraphics[height=3.8cm]{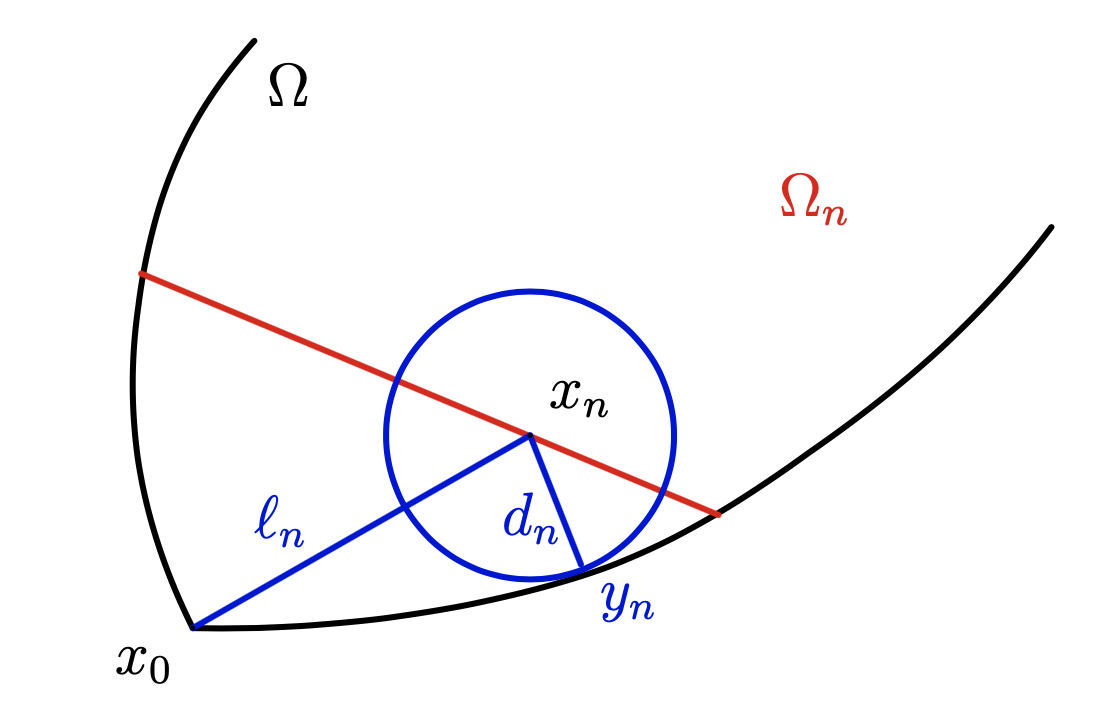}
   \end{center}
\caption{ }
\label{d14f}
\end{figure}

 Following \eqref{d12}, we have
   \begin{equation}\label{d16}
  |\nabla w_{\Om_n}(x_n)|\le C\Big(\frac 1{d_n} \sup _{B(x_n, d_n)} |w_{\Om_n} | +d_n\Big).
    \end{equation}

    \smallskip
 \noindent {\bf Case 1.}  Assume that for a subsequence, $\frac{d_n}{r_n} \ge c >0$, $c$ being some constant. Using $w_{\Om_n}\le w_\Om$ and the decay of $w_\Om$ at the corner point $x_0$, we get
  $$|\nabla w_{\Om_n}(x_n)|\le C\Big(\frac 1{d_n} \sup _{B(x_n, d_n)} |w_{\Om_n} | +d_n\Big) \le C\Big(\frac 1{d_n} \sup _{B(x_n, d_n)} |w_{\Om} | +d_n\Big)$$
  $$\le C\Big(\frac 1{d_n} \sup _{B(x_n, d_n)} (r_n+d_n) \vps (x)+d_n\Big) \le
  C\Big(  \sup _{B(x_n, d_n)} (\frac 1{c}+1) \vps (x)+d_n\Big) \to 0,$$
  which comes from the hypothesis $r_n \le \frac{d_n}c$ and the decay $\vps(x)\to 0$, when $x \to x_0$.

 \smallskip
 \noindent {\bf Case 2.}  Assume now that for a subsequence $\frac{d_n}{r_n} \to 0$. Let $y_n$ be one projection of $x_n$ on $\partial \Om$ and $\PP_n$ a tangent hyperplane to $\Om$ at $y_n$, $\PP_n^+$ being the half space containing $\Om$. Defining $B^+(y_n, r_n+d_n)= B(y_n, r_n+d_n)\cap \PP^+$, we introduce the following barrier function
$$-\Delta \varphi =1 \mbox { in } B^+(y_n, r_n+d_n), \varphi = w_\Om \mbox{ on } \partial B^+(y_n, r_n+d_n).$$
Above, $w_\Om$ is extended by $0$ outside $\Om$ (see Figure \ref{d15}).

\begin{figure}
\begin{center}
  \includegraphics[height=4.5cm]{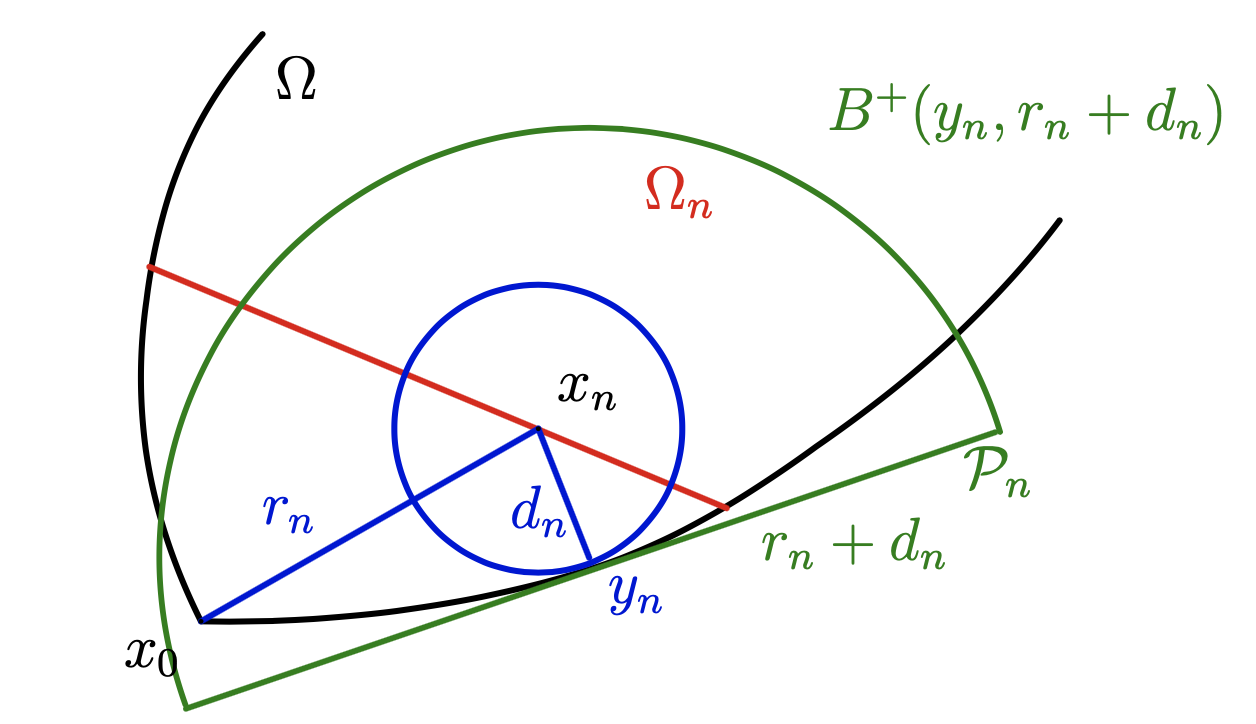}
   \end{center}
\caption{}
\label{d15}
\end{figure}

Relying on \eqref{d13}, we first get an estimate for the gradient of $\varphi$ for every $z \in B^+(y_n, \frac 12 (r_n+d_n))$:
\begin{equation}\label{d17}
|\nabla \varphi (z)| \le  C\Big ( \frac{1}{r_n+d_n } \Big [   \sup _{B^+(y_n, (r_n+d_n))} |\varphi| \Big ]+ (r_n+d_n)\Big).
\end{equation}
But
$$\varphi (z) \le \Big [   \sup_{B (y_n, (r_n+d_n))} w_\Om \Big ] + \frac{\|z-y_n\|^2}{2d}$$
which leads to
$$\varphi (z) \le C (r_n+d_n)  \Big [\sup _{B^+(y_n, (r_n+d_n))}\vps(x)\Big] + \frac{(r_n+d_n)^2 }{2d}$$ for all $z \in B^+(y_n, \frac 12 (r_n+d_n)),$ which gives from \eqref{d17},
$$|\nabla \varphi (z)| \le  C  \sup _{B^+(y_n, (r_n+d_n))}\vps(x).$$
Now, for every $z \in B(y_n, 2d_n)$, we denote by $z_n$ its  projection on $\PP_n$. We have
$$|\varphi (z)|= |\varphi (z) -\varphi (z_n)| = |z-z_n| |\frac{\partial \varphi}{\partial \nu_{\PP_n}} (\xi)| \le C |z-z_n|  \sup _{B^+(y_n, (r_n+d_n))}\vps(x).$$
But
$$w_{\Om_n} (z) \le w_\Om(z)\le \varphi( z) \le C |z-z_n| \sup _{B^+(y_n, (r_n+d_n))}\vps(x)\le Cd_n \sup _{B^+(y_n, (r_n+d_n))}\vps(x).$$
Taking this estimate in \eqref{d16}, we get
$$ |\nabla w_{\Om_n}(x_n)|\le C\Big[\sup _{B^+(y_n, (r_n+d_n))}\vps(x) \Big]   +C d_n \to 0.$$
\end{proof}

\smallskip
\noindent{\bf Boundary differentiability of $w_\Om$.}
Below we describe the differentiability of $w_\Om$ at boundary points and give a characterisation for the value of the slope of the gradient; we give a proof inspired by \cite[Lem. 11.17]{CaffarelliSalsa}, see also \cite[Lem. 4.10]{BLNP26} for a similar method.

\begin{lemma}\label{lem_boundary_differential}
Let $\Om\in \mathcal K_{\R^d}$ be such that $\Om\subset\R^{d-1}\times\R^+$ and $0\in\partial\Om$. For every $r >0$ we define
 $$\alpha(r) =\inf\{\beta >0 : \forall x \in B_r\cap \Om, w_{\Om}(x) \le \beta x_d^+\},$$
and $\alpha:=\lim_{r\to 0}\alpha(r)$ its (monotone) limit. Then
\begin{equation}\label{d61}
w_\Om(x)=\alpha x_d+o_{x(\in\Om)\to 0}(|x|).
\end{equation}
Moreover, for any compact set $K\subset\R^{d-1}\times (0,+\infty)$, we have
\begin{equation}\label{eq_conv_grad_nontangential}
\frac{w_\Om(r \cdot)}{r}\underset{\mathcal{C}^1(K)}{\longrightarrow}\alpha x_d(\cdot)
\end{equation}
as $r\to 0$.
\end{lemma}

\begin{proof}
In this proof we put $L:=\|\nabla w_\Om\|_\infty <+\infty$.

We can assume  $\alpha >0$, otherwise \eqref{d61} is a direct consequence of the definition of $\alpha(r)$.

Clearly, $w_\Om(x)\le  \alpha x_d + o(|x|)$ for $x\in\Om$. Now, if  \eqref{d61} fails to be true, then there exists $\vps >0$ and a sequence $(x^p)_p \subset \Om$, $x^p \to 0$ such that
$$\forall p\in \N, \;\;  w_\Om(x^p)\le \alpha x^p_d-\vps\|x^p\|.$$
Let $y^p=x^p + \frac{\vps}{2L} \|x^p\| e_d$. Then, for $\vps$ small,
$$\|y^p\|\le (1+ \frac{\vps}{2L} ) \|x^p\|\le 2 \|x^p\|$$
and
$$y_d^p \ge  \frac{\vps}{2L} \|x^p\| \ge  \frac{\vps}{4L} \|y^p\|.$$
Consequently, $y^p$ belongs to the cone $\{x_d \ge \frac{\vps}{4L} \|x\|\}$. We have
\begin{equation*}
w_\Om(y^p) \le w_\Om(x^p)+ L\|y^p-x^p\|\le \alpha x_d^p-\vps\|x^p\|+ L\frac{\vps}{2L}\|x^p\|
\le \alpha y_d^p-\frac{\vps}{4}\|y^p\|,
\end{equation*}
where we have used our assumption on $x^p$ in the second inequality.
Let $r_p = \|y^p\|$, $z_p= \frac{1}{r_p} y^p$ and denote $w_p(x)= \frac{1}{r_p}w_\Om(r_px)$. We have
$$\forall x \in B(z_p, c_1\vps), \;\; w_p(x) \le \alpha x_d -\frac{\vps}{8},$$
where $ c_1= \frac{1}{8(\alpha+L)}$. We also have that
$$B(z_p, c_1\vps) \subset (B _\frac{19}{10}\sm B _\frac{1}{10}) \cap \{x_d \ge \frac{\vps}{8L}\}.$$

Let $h_p: (B_2 \cap \{x_d>0\}) \sm \ov B(z_p, c_1\vps) \to [0,1]$ and $\chi_p:B_2\cap\{x_d>0\}\to \R$ be the solutions of
$$
\begin{cases}
\Delta h_p=0 & \text{ in } (B_2 \cap \{x_d>0\})\sm \ov B(z_p, c_1\vps)\\
h_p=0&\text{ in }\partial (B_2 \cap \{x_d>0\} )\\
h_p=1&\text{ in }\partial B(z_p, c_1\vps)
\end{cases},\
\begin{cases}
\Delta \chi_p=0 &\text{ in } B_2 \cap \{x_d>0\}\\
\chi_p=4&\text{ in }
\partial B_2 \cap   \{x_d>0\}\\
\chi_p=0&\text{ in }B_2 \cap   \{x_d=0\}
\end{cases}.
$$
We have the following elementary estimates: for some constants $C_1(\vps,L)<\infty,\,c_2(\vps,L)>0,\,c_3(\vps,L)<\infty,$
we have that
$$h_p(x)\leq C_1(\vps,L)x_d,\ \forall x\in (B_2 \cap \{x_d>0\}) \sm \ov B(z_p, c_1\vps)$$
and
$$h_p(x)\geq c_2(\vps,L)x_d,\ \chi_p(x)\leq c_3(\vps,L)x_d,\ \forall x\in B_\frac{1}{10}\cap\{x_d>0\}.$$
For any $x\in (B_2 \cap \{x_d>0\})\sm \ov B(z_p, c_1\vps)$, let
$$v_p(x)=w_p(x)+ \frac{\vps}{16M}h_p(x) -\alpha(2r_p)x_d + \frac 12 r_p (x_d^2-\chi_p(x)),$$
where $M \ge 1$ will be fixed below. We have $\Delta v_p=0$ in $\om_p$, and
\begin{itemize}
\item[(i)] $\forall x \in \partial B_2 \cap \frac{1}{r_p} \Om$, $$v_p(x)= \underbrace{w_p(x) -\alpha(2r_p) x_d}_{\leq 0}+ \frac{\vps}{16M}\underbrace{h_p(x)}_{=0}+ \frac 12 \underbrace{(x_d^2-\chi_p(x))}_{\leq 0}\le 0.$$
\item[(ii)] $\forall x \in  B_2 \cap \partial \frac{1}{r_p} \Om$, $v_p(x)\le x_d(-\alpha + \frac{\vps}{16M}C_1(\vps, L)+r_p)$. For any sufficiently large $p$, the right-hand side above is negative when $M$ is sufficiently large depending only on $\alpha, \vps, L$.
\item[(iii)] $\forall x \in \partial B(z_p, c_1\vps)$, $v_p(x) \le -\frac{\vps}8+ \frac{\vps }{16M}+2r_p\le -\frac\vps{16}$, for $p$ large.
\end{itemize}
By the maximum principle we conclude that $v_p \le 0$ in $(B_2 \cap \{x_d>0\})\sm \ov B(z_p, c_1\vps)$.  This implies that $\forall x\in B_{\frac {1}{10}}$,
\begin{align*}
w_p(x) &\le \alpha(2r_p)x_d-\frac{\vps}{8M}h_p(x)+ \frac 12 r_p(\chi_p(x)-x_d^2)\\
&\le  \left (\alpha(2r_p) -\frac{c_2(\vps,L)\vps}{8M} + \frac{c_3(\vps,L)r_p}2\right ) x_d^+.
\end{align*}
This contradicts the minimality of $\alpha\left(\frac{r_p}{10}\right)$ for $p \to +\infty$, and proves \eqref{d61}.
Formula \eqref{eq_conv_grad_nontangential} is obtained by a blow-up argument: indeed for any sufficiently small $r$, in a neighbourhood of $K$ in $\R^{d-1}\times (0,+\infty)$ the function $\frac{w_\Om(r\cdot)}{r}+r\frac{x_d^2}{2}$ is harmonic and converges uniformly to $\alpha x_d$, so $\nabla\left(\frac{w_\Om(r\cdot)}{r}+r\frac{x_d^2}{2}\right)$ converges uniformly (in $K$) to $\alpha e_d$.
\end{proof}

\begin{lemma}\label{lem_conv_gradient_subdomain}
Let $\Om \in \mathcal K_{\R^d}$ be such that $\Om\subset\R^{d-1}\times\R_+$ and $0\in\partial\Om$. Let $(\delta_n)_{n\in\N}$ be a sequence of positive numbers converging to $0$. Then there exists a sequence of positive numbers $(\xi_n)_{n\in\N}$ converging to $0$ such that
\begin{equation}\label{eq_conv_grad_subdomains}
\frac{w_{\Om}(0',\xi_n)}{\xi_n}\geq \frac{\partial w_{\Om\cap\{x_d\geq \delta_n\}}}{\partial x_d}(0',\delta_n)+o_{n\to +\infty}(1).
\end{equation}
\end{lemma}

\begin{proof}
In the  proof, we denote by  $\Om_n:=\Om\cap\{x_d>\delta_n\}$.

If the origin is an angular point, then \eqref{eq_conv_grad_subdomains} is automatically verified: the right-hand side converges to zero by Lemma \ref{d04}. Thus from now on we assume $0$ is not an angular point.

Let $(a_n)_{n\geq 0}$  be the inradii of $\Om\cap\{x_d=\delta_n\}$  i.e. the largest value such that
$$\Om\cap\{x_d=\delta_n\}\supset B_{a_n}^{d-1}\times\{\delta_n\}.$$
where $B_{r}^{d-1}$ denotes the ball of radius $r$ in $\R^{d-1}$.
Since the origin is not an angular point, we have that
$$\frac{\delta_n}{a_n}\underset{n\to +\infty}{\longrightarrow}0.$$
We fix $b_n:=\sqrt{a_n\delta_n}$ such that $\frac{b_n}{a_n}\to 0$ and $\frac{b_n}{\delta_n}\to \infty$. Let
$$g_n(x):=\frac{1}{b_n}w_{\Om_n}\left((0',\delta_n)+b_n x\right)+b_n\frac{x_d^2}{2}-\frac{\partial w_{\Om_n}}{\partial x_d}(0',\delta_n)x_d.$$
For any $R>0$ and for any sufficiently large $n$, $g_n$ is harmonic in $B_R\cap\{x_d>0\}$ with $g_n=0$ in $B_R\cap\{x_d=0\}$, and $\frac{\partial g_n}{\partial x_d}(0)=0$. Moreover
$$\limsup_{n\to +\infty}\Vert\nabla g_n\Vert_{L^\infty(B_R\cap\{x_d>0\})}\leq \limsup_{n\to +\infty}\Vert\nabla w_{\Om_n}\Vert_{\infty}\leq \sqrt{2\Vert w_\Om\Vert_\infty}.$$
By elliptic boundary regularity, $(g_n)_{n\in\N}$ is compact in $\mathcal{C}^2_{\mathrm{loc}}(\R^{d-1}\times\R^+)$, and any accumulation point $g$ is a  globally Lipschitz harmonic function that vanishes in $\{x_d=0\}$, and is such that $\frac{\partial g}{\partial x_d}(0)=0$: by Liouville's Theorem $g=0$, so the full sequence $(g_n)_{n\in\N}$ converges locally uniformly to  $0$.

The convergence of $g_n(0,1)$ to $0$ gives that
$$\frac{w_{\Om_n}(0',\delta_n+b_n)}{b_n}-\frac{\partial w_{\Om_n}}{\partial x_d}(0',\delta_n)+\frac{b_n}{2}\underset{n\to +\infty}{\longrightarrow}0.$$
By the maximum principle we have that $w_{\Om}\geq w_{\Om_n}$, and $\frac{\delta_n}{b_n}\to 0$ so letting $\xi_n=\delta_n+b_n$ we have
$$\frac{w_{\Om}(0',\xi_n)}{\xi_n}-\frac{\partial w_{\Om_n}}{\partial x_d}(0',\delta_n)\geq \frac{b_n}{b_n+\delta_n}\frac{w_{\Om_n}(0',\delta_n+b_n)}{b_n}-\frac{\partial w_{\Om_n}}{\partial x_d}(0',\delta_n)\geq o_{n\to +\infty}(1).$$
\end{proof}

Below we give the proof of the main result.
\begin{proof}[Proof of Theorem \ref{th_existence_failpoint}]
We prove each point separately.

\noindent 1. This is a consequence of Lemma \ref{lem_boundary_differential}.

\noindent 2. Without loss of generality, we may assume that $\Om \subset \R^{d-1}\times \R^+$, $\ov{x}=0$ and $\nu_{\ov{x}}=-e_d$.
\smallskip
By  \eqref{eq_conv_grad_nontangential} of Lemma \ref{lem_boundary_differential}, we immediately have that

$$\nabla w_\Om(0)=\lim_{t(>0)\to 0}\nabla w_\Om(0',t).$$
Hence
\begin{equation}\label{eq_conv_gradient_verticalline}
|\nabla w_\Om(0)|=\lim_{t(>0)\to 0}\frac{\partial w_\Om}{\partial x_d}(0',t)\leq \limsup_{x(\in\Om)\to 0}|\nabla w_\Om(x)|.
\end{equation}

\smallskip

We now prove that
\begin{equation}\label{d63}
|\nabla w_\Om(0)| \ge  \limsup_{x(\in\Om)\to 0} |\nabla w_\Om(x)|.
\end{equation}
Consider some sequence $(y_n)_{n\in\N}$ of points in $\Om$ converging to the origin such that
$$\lim_\nif |\nabla w_\Om(y_n)|= \limsup_{x(\in\Om)\to 0} |\nabla w_\Om(x)|.$$
If $0$ is an angular point, then $|\nabla w_\Om(0)|=0$ and Corollary \ref{cor_gradientcornerpoint} implies the assertion. From now on we assume that $0$ is not an angular point.

We first prove that there exists a sequence of chord domains
$$\Om_n:= \{ x\in \Om : x_d >\delta_n\},$$
with $\delta _n \to 0_+$, such that
$$\liminf_\nif|\nabla w_{\Om_n} (0', \delta_n) | \ge \lim_\nif |\nabla w_\Om(y_n)|.$$
Here, we use the power concavity of the torsion function. Let us first define
$$\Om_n^{1} :=  \{ x\in \Om : w_\Om (x) > w_\Om(y_n)\}.$$
The set  $\Om_n^1$ defined above is convex and smooth. Moreover, since $0$ is not an angular point
$$\nu_{\Om_n^1} (y_n) \to \nu_{\Om} (0)=-e_d.$$
Let $\mathcal{R}_n$ be a direct rigid movement (consisting of a rotation $\Vec{\mathcal{R}_n}$ and a translation) such that $\mathcal{R}_n(y_n)=0$ and $\Vec{\mathcal{R}_n}(\nu_{\Om_n^1}(y_n))=-e_d$. Let $p \in \Om \cap (\{0\}^{d-1}\times (0,+\infty))$ be fixed, and let $t_n \in (0,1]$ be the largest number such that
$$\Om_n^2: =(1-t_n)p+ t_n \mathcal{R}_n(\Om_n^1) \subset \Om.$$
Then $t_n \to 1$ since $y_n\to 0$ and $\nu_{\Om_n^1}(y_n)\to -e_{d}$.
Hence there exists a sequence of positive numbers $\delta_n$ such that $\delta_n\to 0$ and
$$(0', \delta_n)\in \partial \Om_n^2 \cap \{0\}^{d-1}\times (0,+\infty).$$
We have that
$$|\nabla w_{\Om_n^2} (0', \delta_n) | \to \limsup_{x(\in\Om)\to 0} |\nabla w_\Om(x)| \mbox{ for } \nif.$$
We now let
$$\Om_n := \{ x\in \Om : x_d > \delta_n\}.$$
Since $\Om_n\supseteq\Om_n^2$ and $(0',\delta_n)$ belongs to both sets, we get by the maximum principle that
\begin{equation}\label{d64}
\limsup_\nif |\nabla w_{\Om_n} (0', \delta_n) |\ge  \limsup_{x(\in\Om)\to 0} |\nabla w_\Om(x)|.
\end{equation}
Applying Lemma \ref{lem_conv_gradient_subdomain} to the sequence $\Om_n$,  there exists a sequence of positive numbers $(\xi_n)_{n\in\N}$ converging to $0$ such that
$$\lim_\nif \frac{w_\Om(0', \xi_n)}{\xi_n}\geq  \limsup_{x(\in\Om)\to 0} |\nabla w_\Om(x)|.$$
By Lemma \ref{lem_boundary_differential}, we conclude that
$$|\nabla w_\Om(0)|=\lim_\nif \frac{w_\Om(0', \xi_n)}{\xi_n}\geq  \limsup_{x(\in\Om)\to 0} |\nabla w_\Om(x)|.$$
This completes the proof of Theorem \ref{th_existence_failpoint} item 2.

\noindent 3. Let $\ov{x}\in\partial\Om$ and $(x_n)_{n\in\N}$ be a sequence in $\partial \Om$ that converges to $\ov{x}$. By the previous point, for each $n,$ we may find some $y_n\in\Om$ such that
$$\|x_n-y_n\|\le \frac 1n \mbox{  and  } \Big ||\nabla w_\Om(y_n)|-|\nabla w_\Om(x_n)| \Big | \le \frac 1n.$$
By item 2,
$$|\nabla w_\Om(\ov{x})|\geq \limsup_{\nif} |\nabla w_\Om(y_n)| = \limsup_{\nif}|\nabla w_\Om(x_n)|.$$
This completes the proof of Theorem \ref{th_existence_failpoint} item 3.

\noindent 4. By subharmonicity of $|\nabla w_\Om|^2$ in $\Om$, $|\nabla w_\Om|$ reaches its supremum in $\Om$ for some sequence $(x_n)_{n\in\N}$ in $\Om$ approaching the boundary $\partial\Om$. Up to extraction we may assume $(x_n)_{n\in\Om}$ converges to some limit point $\ov{x}\in\partial\Om$, and by item 2 we have that
$$|\nabla w_\Om(\ov{x})|\geq \lim_{n\to +\infty}|\nabla w_\Om(x)|= \Vert\nabla w_\Om\Vert_{\infty}.$$
Furthermore, by item 2, we have for any $y\in\partial\Om$:
$$|\nabla w_\Om(y)|= \limsup_{x(\in\Om)\to y}|\nabla w_\Om(x)|.$$
This completes the proof of Theorem \ref{th_existence_failpoint} item 4.
\end{proof}

\section{Stability of the maximum stress to geometric perturbations}\label{sec_stability}

The aim of this section is to prove that the maximum stress of a convex set is continuous under the Hausdorff complementary topology.

\begin{proof}[Proof of Theorem \ref{th_stability_maxstress}]
If $\Om=\emptyset$, from the $H^c$-convergence we get that the width of $\Om_n$ converges to $0$. This implies that
$\|w_{\Om_n}\|_\infty \to 0$ so that  by the second inequality in \eqref{d00} $\|\nabla w_{\Om_n}\|_\infty \to 0$. This implies the assertion.

Assume now that $\Om\not =\emptyset$. Then, the sets $\Om_n$ are uniformly bounded and lie in a neighborhood of $\Om$, that we denote by $D$. It is classical that $\Om_n$ $\gamma$-converges to $\Om$ so that
$$w_{\Om_n} \to w_\Om \mbox{ strongly in } H^1_0(D).$$
This implies the lower semicontinuity
$$\|\nabla w_{\Om}\|_\infty\le \liminf_\nif \|\nabla w_{\Om_n}\|_\infty.$$

\smallskip

It remains to prove the upper semicontinuity. In order to do so we provide a series of preliminary results. The proof follows similar lines as the proof of Theorem \ref{th_existence_failpoint} item 3.
Without loss of generality, we may assume that $\Om_n \sq \Om$. Indeed, we fix a centre point in $\Om$ and we take $t_n \in (0,1]$ such that $t_n \Om_n \sq \Om$ and $t_n$ is maximal with this property. Then $t_n\to 1$ as a consequence of the Hausdorff convergence $\Om_n \to \Om$, and
$$\lim_\nif \|\nabla w_{\Om_n}\|_\infty=\lim_\nif \|\nabla w_{t_n\Om_n}\|_\infty.$$
  We relabel the new sequence as $(\Om_n)_n$.

Let $x_n\in\partial\Om_n$ be a fail point of $\Om_n$ (which exists by Theorem \ref{th_existence_failpoint}). Without loss of generality we may assume, up to a subsequence and a rigid transformation (a composition of a rotation $\Vec{\mathcal{R}_n}$ and a translation), that
$$x_n\to 0,\ \Om\subset\R^{d-1}\times\R^+.$$

If $0$ is an angular point, then $|\nabla w_{\Om_n}(x_n)|\to 0$ by an application of Lemma \ref{d04} to the chord domain of $\Om$ that is tangent to $\Om_n$ at $x_n$, so there is nothing to prove.

We assume from now on that the origin is not an angular point, which guarantees that $\nu_{\Om}(0)=-e_d$ and $\nu_{\Om_n}(x_n)\to -e_d$.

 Let $\mathcal{R}_n$ be a rigid transformation such that $\mathcal{R}_n(x_n)=0$ and $\Vec{\mathcal{R}_n}(\nu_{\Om_n}(x_n))=-e_d$. Let $p \in \Om \cap (\{0\}^{d-1}\times (0,+\infty)$ be fixed, and let $t_n \in (0,1]$ be the largest number such that
$$\Om_n': =(1-t_n)p+ t_n \mathcal{R}_n(\Om_n') \sq \Om.$$
Then $t_n \to 1$ since $x_n\to 0$ and $\nu_{\Om_n}(x_n)\to -e_{d}$. This implies that
$$\limsup_{n\to +\infty}|\nabla w_{\Om_n'}(0',\delta_n)|=\limsup_{n\to +\infty}|\nabla w_{\Om_n}(x_n)|=\limsup_{n\to +\infty}\Vert \nabla w_{\Om_n}\Vert_\infty,$$
where $(\delta_n)_{n\in\N}$ is the sequence of positive numbers defined by $\mathcal{R}_n(x_n)=(0',\delta_n)$. Since $\nu_{\Om_n'}(0',\delta_n)=-e_d$ and $\Om_n'\subset\Om$, we have the inclusion
$$\Om_n'\subset\Om\cap\{x_d>\delta_n\}$$
so
$$\limsup_{n\to +\infty}|\nabla w_{\Om\cap\{x_n>\delta_n\}}(0',\delta_n)|\geq\limsup_{n\to +\infty}\Vert \nabla w_{\Om_n}\Vert_\infty.$$
By an application of Lemma \ref{lem_conv_gradient_subdomain} to the sequence of chord domains $\Om\cap\{x_d>\delta_n\}$, there exists a sequence of positive numbers $(\xi_n)_{n\in\N}$ converging to $0$ such that
$$
\limsup_{n\to +\infty}\frac{w_\Om(0',\xi_n)}{\xi_n}\geq \limsup_{n\to +\infty}\Vert \nabla w_{\Om_n}\Vert_\infty.
$$
By Lemma \ref{lem_boundary_differential}, the left-hand side is $|\nabla w_\Om(0)|$, which proves upper semicontinuity:
$$\Vert \nabla w_{\Om}\Vert_\infty\geq \limsup_{n\to +\infty}\Vert \nabla w_{\Om_n}\Vert_\infty.$$

\end{proof}
\begin{remark}\rm
The example mentioned in the introduction is relevant. Consider the sequence of regular $n$-gons inscribed in the unit circle. This sequence converges to the unit disk in the Hausdorff metric, while the corresponding torsion functions converge strongly in $H^1$ to the radial torsion function of the disc. Nevertheless, the behaviour of their gradients at the boundary is markedly different from that of the limiting torsion function. Indeed, along the boundary of each $n$-gon, the gradient oscillates: it vanishes at every vertex, while the maximum of its norm is attained at the midpoint of each edge.  Moreover, the gradients are  H\"older continuous up to the boundary: $\nabla w_{\Om_n} \in C^{0, \frac{\delta_n}2}(\ov \Om_n)$ since the torsion function belongs to $H^{2, \delta_n}(\Om_n)$, for some $\delta_n >0$, with $\delta_n \to 0$. Thus, although each gradient is continuous up to the boundary, this continuity is not uniform with respect to $n$. Nevertheless, the gradients converge in $L ^2$
to the gradient of the torsion function on the disk, whose norm is constant along the boundary. The stability result shows that,  despite this lack of uniform continuity, the maximum of the norm remains continuous under this convergence.\end{remark}

\begin{remark}\rm[Unbounded domains] \rm Assume $\Om$ is an unbounded convex set with torsion function $w_{\Om}$. If the gradient of its torsion function is bounded, then the width of $\Om$ is finite. In general,   if $\Om_n$ are open, convex sets which converge locally in the Hausdorff metric to $\Om$, then the $L^\infty$ norms of the torsion functions of $\Om_n$ and of their gradients will not generally converge.
Take for instance $\Om= (0, \infty) \times (0,1)$ and
$$\Om_n = \{(x,y) : 0<x, 0<y<1+ \frac {x\wedge n}n  \}.$$
 \end{remark}

\section{A free boundary problem: inward maximisers for the gradient magnitude of the torsion function}\label{sec_inwardmax}

The occurrence, under various constraints, of flat portions in the boundary of the maximizers of the gradient magnitude of torsion functions  is related to the following general notion of inward maximisers. 

\begin{definition}\label{def_inward_max}
Let $\Om\in \mathcal K_{\R^d}$. We say that $\Om$ is an inward maximiser for the maximum stress if, for any $\Om'\in \mathcal K_{\R^d}$ such that $\Om'\subset\Om$, we have
$$\Vert \nabla w_\Om\Vert_\infty\geq \Vert \nabla w_{\Om'}\Vert_\infty.$$
\end{definition}
Note that there is no loss of generality in assuming the above condition only for chord subdomains of $\Om$, since one can consider the chord domain supported at a fail point of $\Om'$.
\begin{remark}\rm
Let $D$ be an open, bounded set of $\R^d$. Let $\Om$ be a solution of
$$\max\{ \|\nabla w_\Om\|_\infty : \Om\in \mathcal K_{\R^d}, \Om\sq D\}.$$
Then, $\Om$ is a inward maximiser for the maximum stress.
\end{remark}
The following proposition justifies that this notion applies to maximisers of the maximum stress under a wide variety of geometric constraints.
\begin{proposition}\label{d120}
Let $G:\mathcal K_{\R^d}\to(0, +\infty)$ be such that $G$ is nondecreasing with respect to inclusions and positively homogeneous with respect to rescaling.
Assume $\Om$ is a solution of
$$\max\{ \|\nabla w_\Om\|_\infty :  \Om\in \mathcal K_{\R^d}, G(\Om)=m\}$$
for some $m>0$. Then $\Om$ is an inward maximiser for the maximum stress.
\end{proposition}
 By changing $G$ into $1/G$ we see that this also applies to functions which are nonincreasing with respect to inclusions, and negatively homogeneous with respect to rescaling. Examples of admissible functions $G$ are the quermassintegrals  (e.g. volume, perimeter), diameter, inradius, circumradius, torsion, first eigenvalue of the Dirichlet Laplacian, etc.
\begin{proof}
Let $\tilde{\Om}\in \mathcal K_{\R^d}$ be such that $\tilde{\Om}\subset\Om$. Since $G(\tilde{\Om})\leq G(\Om)$, by homogeneity there exists some $t\geq 1$ such that $G(t\tilde{\Om})=G(\Om)$. By optimality of $\Om$, we have
$$\Vert \nabla w_\Om\Vert_\infty\geq \Vert \nabla w_{t\tilde{\Om}}\Vert_\infty=t\Vert \nabla w_{\tilde{\Om}}\Vert_\infty\geq \Vert \nabla w_{\tilde{\Om}}\Vert_\infty.$$
\end{proof}

Before proving Theorem \ref{th_inward_max}, we give a separate technical lemma on the local nondegeneracy of fail points.

\begin{lemma}\label{lem_isolatedfailpoint}
Let $w\in H^1(B_1\cap\{x_d>0\})$, such that
$$
\begin{cases}
-\Delta w=1 & \text{ in } (B_1\cap\{x_d>0\})\\
w=0 & \text{ in }(B_1\cap\{x_d=0\})\\
\end{cases}
$$
Assume that $w$ is positive, $\frac{1}{2}$-concave, and that $\partial_{x_d}w$ reaches its maximum at the origin. Then either this maximum is strict, or $e\cdot\nabla w=0$ for some nonzero vector $e\in \R^{d-1}\times\{0\}$.
\end{lemma}
Note that in dimension $d=2$, this lemma is direct by unique continuation at the boundary.
\begin{proof}
We assume $d\geq 3$ and prove the following stronger conclusion: if $\partial_{x_1,x_1,x_d}w(0)=0$ and $\partial_{x_1}w\neq 0$, then a contradiction arises. This is sufficient, by the concavity of $\sqrt{\partial_{x_d}w}$ in $B_1^{d-1}\times\{0\}$.

By analyticity of $w$ we write
$$w(x)=\sum_{k\in\N}w_k(x)$$
where $w_k$ is $k$-homogeneous. We directly have $w_0=0$, $w_1(x)=\partial_{x_d} w(0) x_d$. Observing that $w_2(x)+\frac{1}{2}x_d^2$ must be a combination of $x_jx_d$ for $j<d$ by harmonicity, since $\partial_{x_d}w$ is maximal at the origin we must have $w_2(x)=-\frac{1}{2}x_d^2$. For any $i<d$, we obtain that $$\partial_{x_i}w(x)=\mathcal{O}(|x|^2),\ \partial_{x_i,x_i}w(x)=\mathcal{O}(|x|).$$
There exists some symmetric matrix $A\in S_{d-1}(\R)$ such that
$$w_3(x',x_d)=x_d\left(\frac{1}{3}\mathrm{tr}(A)x_d^2-( x',Ax')\right)$$
By maximality of $\partial_{x_d} w$ at the origin, $A$ must be nonnegative. We assumed above that
$$\partial_{x_1,x_1,x_d}w(0)=-2A_{1,1}=0$$
In particular this implies $A_{1,i}=0$ for all $i<d$, so $\partial_{x_1}w_{3}=0$. Since $\partial_{x_1}w$ is not identically zero, we may define $\kappa\geq 4$ to be the smallest integer such that $\partial_{x_1} w_{\kappa}$ is nonzero. By $\frac{1}{2}$-concavity of $w$ we have $$\partial_{x_1,x_1}w\leq \frac{1}{2}w^{-1}(\partial_{x_1}w)^2=\mathcal{O}(|x|^{2\kappa-3}).$$
We remind that $\partial_{x_1,x_1}w_\ell$ must change sign (if it is non-zero) for $\ell\geq 4$ by Hopf's lemma since it is $(\ell-2)$-homogeneous, thus the previous inequality implies $\partial_{x_1,x_1}w_{\ell}=0$ for $\ell=4,\hdots,2\kappa-2$. Thus
$$\partial_{x_1,x_1}w(x)=\mathcal{O}(|x|^{2\kappa-3}).$$
In particular since $2\kappa-2\geq \kappa$, we obtain $\partial_{x_1,x_1}w_\kappa=0$, so
$$\partial_{x_1}w_\kappa(x)=v(x_2,\hdots,x_d)$$
for some $(\kappa-1)$-homogeneous non-zero harmonic function $v$. We may fix some $i\in\{2,\hdots,d-1\}$ such that $\partial_{x_i}v$ is nonzero. Then by the concavity of $w^\frac{1}{2}$, we have
$$\left|\partial_{x_1,x_i}w-\frac{1}{2}w^{-1}\partial_{x_1}w\ \partial_{x_i}w\right|^2\leq \left|\partial_{x_1,x_1}w-\frac{1}{2}w^{-1}(\partial_{x_1}w)^2\right|\cdot\left|\partial_{x_i,x_i}w-\frac{1}{2}w^{-1}(\partial_{x_i}w)^2\right|$$
Thus by the previous estimates we have
$$\left|\partial_{x_1,x_i}w-\frac{1}{2}w^{-1}\partial_{x_1}w \partial_{x_i}w\right|=\mathcal{O}(|x|^{\kappa-1})$$
Since $w^{-1}\partial_{x_1}w\partial_{x_i}w=\mathcal{O}(|x|^{\kappa})$, this implies
$$\partial_{x_i}v(x_2,\hdots,x_d)=\mathcal{O}(|x|^{\kappa-1})$$
The left-hand side is a non-zero homogeneous polynomial of degree $\kappa-2$, hence a contradiction.
\end{proof}

\begin{proof}[Proof of Theorem \ref{th_inward_max}]
Assume $\Om \subset \R^{d-1}\times (0, +\infty)$ and $0 \in \partial \Om$ is a fail point.  Since $0$ cannot be angular, the normal field $\nu$ of $\partial \Om$ is continuous at $0$.

For every $\vps >0$ small, we let $\tilde x \in \R^{d-1}$ be such that $(\tilde x(\vps), \vps)$  is the point where
$$x\in\R^{d-1}\mapsto w_\Om(x,\eps)$$
reaches its maximum on $\H_\vps \cap \Om$, where $\H_\vps = \{x_d =\vps\}$. By the concavity of $w_\Om ^\frac 12$, the point $(\tilde x(\vps), \vps)$  is a boundary point of the convex level set $\{x : w_\Om(x)>w_\Om (\tilde x(\vps), \vps)\}$, which is a subset of $ \H_\vps^+= \{ x_d >\vps\}$ for sufficiently small $\vps$.
The point $(\tilde x(\vps), \vps)$ is uniquely defined and the mapping $\vps \to \tilde x(\vps)$ is continuous.

Let us denote
$$ M(\vps) =w_\Om (\tilde x(\vps), \vps).$$
$M$ is continuously differentiable with
$$M'(\vps)=  \frac{\partial w_\Om}{\partial x_d} (\tilde x(\vps), \vps)$$
(note that the differentiability of $\tilde{x}(\vps)$ is not required for this). Likewise, we have in the distributional sense
$$M''(\vps)\geq\frac{\partial^2 w_\Om}{\partial x_d^2} (\tilde x(\vps),\vps)$$
as may be observed by taking the discrete differential
\begin{align*}
\frac{M(\vps+h)-2M(\vps)+M(\vps-h)}{h^2}&\geq \frac{w_\Om(\tilde{x}(\vps),\vps+h)-2w_\Om(\tilde{x}(\vps),\vps)+w_\Om(\tilde{x}(\vps),\vps-h)}{h^2}\\
&= \frac{\partial^2 w_\Om}{\partial x_d^2} (\tilde x(\vps),\vps)+o_{h\to 0}(1).
\end{align*}
Using the equation $\Delta w_\Om=-1$, and the fact that $\tilde{x}\mapsto w(\tilde{x},\vps)$ reaches its maximum in $\tilde{x}=\tilde{x}(\vps)$ so $\frac{\partial^2 w_\Om}{\partial x_j^2}(\tilde{x}(\vps),\vps)\leq 0$ for every $j\in\{1,\hdots,d-1\}$, we get
$$\frac{\partial^2 w_\Om}{\partial x_d^2} (\tilde x(\vps),\vps)\geq -1.$$
Thus we have, for any $0<s<t$,
\begin{equation}\label{eq_equadiffM}
M'(s)\leq M'(t)+(t-s).
\end{equation}
Let $\vps>0$. Let $\delta\in (0,\vps)$. By the definition of $\tilde{x}(\delta)$ we have
$$M(\delta)\geq w_{\Om}(0,\delta),$$
and by the mean value theorem, there exists $\xi(\delta)\in (0,\delta)$ such that
$$M(\delta)=\delta M'(\xi(\delta))$$
So applying \eqref{eq_equadiffM} with $(s,t)=(\xi(\delta),\vps)$ gives
$$M'(t)+t\geq\xi(\delta)+\frac{w_\Om(0,\delta)}{\delta}.$$
Taking $\delta\to 0$ for a fixed $\vps>0$ and using Lemma \ref{lem_boundary_differential} on the last term, we obtain, for all sufficiently small $\vps>0$, that
\begin{equation}\label{eq_estimateM'(t)}
M'(\vps)+\vps\geq \Vert\nabla w_\Om\Vert_{\infty}.
\end{equation}

Let us denote $\Om_\vps= \Om \cap \{x_d >\vps\}$. Then $w_{\Om_\vps}-w_\Om$ is harmonic in $\Om_\vps$, negative in $\Om_\vps$ and vanishes on $\partial \Om \cap \partial \Om_\vps$. Its minimal value   in $\Om_\vps$   is attained at $(\tilde x(\vps), \vps)$.
As a consequence of the Hopf principle, we get
$$\frac{\partial}{\partial x_d}(w_{\Om_\vps}-w_\Om) |_{(\tilde x(\vps), \vps)} >0.$$
We shall estimate this derivative using the harmonic function
$$g_\vps: \R^{d-1} \times (\vps, +\infty) \to \R,\;  g_\vps|_{\R^{d-1}\times \{\vps\}}= 1_{\H_\vps\cap \Om},$$
given by the Poisson kernel in the half space. Based on  the fact that
$$w_{\Om_\vps}-w_\Om \ge - w_{\Om}(\tilde x(\vps), \vps) g_\vps\text{ in }\Om_\vps$$
and
$$(w_{\Om_\vps}-w_\Om) (\tilde x(\vps), \vps)= -w_{\Om}(\tilde x(\vps), \vps) g_\vps (\tilde x(\vps), \vps),$$
we get
\begin{equation}\label{dba10}
\frac{\partial}{\partial x_d}(w_{\Om_\vps}-w_\Om) |_{(\tilde x(\vps), \vps)}  \ge w_{\Om}(\tilde x(\vps), \vps)\left(-\frac{\partial g_\vps}{\partial x_d}(\tilde x(\vps), \vps)\right).
\end{equation}
We now find a lower bound for the last factor. To simplify the computations, let $K_\vps= -(\tilde x(\vps), \vps)+\H_\vps\cap \Om$ and $h_\vps: \R^{d-1}\times (0, +\infty) \to \R$, $h_\vps(\tilde x, x_d)= g_\vps(\tilde x+ \tilde x(\vps), x_d+\vps)$ given by the Poisson kernel
$$h_\vps (x)= \frac{2x_d}{d\om_d} \int_{K_\vps}\frac{1}{|x-y|^d} \d y.$$
 We define, for any $\theta\in\mathbb{S}^{d-2}$,
$$\rho_{K_\vps}(\theta)=\sup\left\{r>0:(r\theta,0)\in K_\vps\right\}.$$
Then
\begin{align*}
\frac{h_\vps(0,0)-h_\vps(0,\delta)}{\delta}&=\frac{2}{d\om_d\delta}\left(1-\int_{K_\eps}\frac{\delta}{(\delta^2+|y|^2)^\frac{d}{2}}\d y\right)=\frac{2}{d\om_d}\int_{\R^{d-1}\setminus K_\eps}\frac{1}{(\delta^2+|y|^2)^\frac{d}{2}}\d y,
\end{align*}
so, taking the limit $\delta\to 0$, we have
\begin{align}\label{dba11}
-\frac{\partial h_\eps}{\partial x_d}(0,0)&=\frac{2}{d\om_d}\int_{\R^{d-1}\setminus K_\eps}|y|^{-d}\d y=\frac{2}{d\om_d}\int_{\mathbb{S}^{d-2}}\rho_{K_\eps}(\theta)^{-1}\d\theta.
\end{align}
Let $r_\vps=\inf_{\theta\in\mathbb{S}^{d-2}}\rho_{K_\vps}(\theta)$. There exists some $y_\vps\in \partial K_\vps$ such that $|y_\vps|=r_\vps$. Then, by convexity of $K_\vps$, we have $\rho_{K_\eps}(\theta)\leq \frac{r_\eps}{\left(\frac{y_\eps}{|y_\eps|},\theta\right)}$ for any $\theta\in\mathbb{S}^{d-2}$ such that $\left(y_\eps,\theta\right)>0$. So
$$\rho_K(\theta)^{-1}\geq r_\vps^{-1}\left(\frac{y_\eps}{|y_\eps|},\theta\right)_+.$$
We obtain the lower bound:
\begin{align*}
-\frac{\partial h_\eps}{\partial x_d}(0,0)&\geq \frac{2}{d\om_d}\int_{\mathbb{S}^{d-2}}\rho_K(\theta)^{-1}\d \theta\geq \frac{2}{d\om_d}r_\eps^{-1}\int_{\mathbb{S}^{d-2}}\left(\theta,\frac{y_\vps}{|y_\vps|}\right)_+\d\theta=(\pi r_\eps)^{-1}.
\end{align*}
{
Note that in dimension $d=2$, this yields the same constant:
$$
-\frac{\partial h_\eps}{\partial x_d}(0,0)\geq \pi^{-1}\sum_{\theta\in \{0,\pi\}} \rho_{K_\eps}(\theta)^{-1}\geq \pi^{-1} \min(\rho_{K_\eps}(0),\rho_{K_\eps}(\pi))^{-1}.
$$
}

As a consequence, we get
\begin{equation}\label{d102}
 \frac{\partial w_{\Om_\vps}}{\partial x_d}(\tilde x(\vps), \vps)\ge \frac{\partial w_{\Om}}{\partial x_d}(\tilde x(\vps), \vps)+ w_{\Om} (\tilde x(\vps), \vps)(\pi r_\eps)^{-1}.
 \end{equation}
Putting together \eqref{eq_estimateM'(t)} and \eqref{d102}, we get
$$\frac{\partial w_{\Om_\vps}}{\partial x_d}(\tilde x(\vps), \vps)\ge \|\nabla w_\Om\|_\infty -\vps+ w_{\Om} (\tilde x(\vps), \vps)(\pi r_\eps)^{-1}.$$
Since by hypothesis $\Om$ is an inward maximiser for the maximum stress, the left-hand side is bounded by $\Vert\nabla w_\Om\Vert_{\infty}$, so
\begin{equation}\label{d103}
\vps\ge w_{\Om} (\tilde x(\vps), \vps)(\pi r_\eps)^{-1}\geq  w_{\Om} (0, \vps)(\pi r_\eps)^{-1}.
 \end{equation}
Multiplying by $r_\eps/\eps$ and passing to the limit, we get as $\vps\to 0$:
\begin{equation}\label{eq_estinverserhoK}
\liminf_{\eps\to 0}r_\eps \ge \pi^{-1}\|\nabla w_\Om\|_\infty.
\end{equation}

Let $\tilde{x}_0$ be some accumulation point of $(\tilde{x}_0(\eps))_{\eps\to 0}$. Let $K$ be the Hausdorff limit of $(K_\vps)_{\vps\to 0}$, and $r=\liminf_{\vps\to 0}r_\vps$, then $K$ contains $(\tilde{x}_0,0)+B_{r}^{d-1}\times\{0\}$, and by \eqref{eq_estinverserhoK} we have
$$r\geq \pi^{-1}\Vert \nabla w_\Om\Vert_{\infty}.$$
Finally, we prove that we must have $\tilde{x}_0=0$. Assume by contradiction that $\tilde{x}_0\neq 0$, and $(\eps_k)$ is some vanishing sequence such that $\tilde{x}(\eps_k)\to\tilde{x}_0$. For any $k$, there exists some $\delta_k\in (0,\eps_k)$ such that $$\partial_{x_d}w_{\Om}(\tilde{x}(\eps_k),\delta_k)\geq \eps_k^{-1}w_{\Om}(\tilde{x}(\eps_k),\eps_k)\geq \eps_k^{-1}w_{\Om}(0,\eps_k)\to \Vert \nabla w_\Om\Vert_\infty$$
so by upper semi-continuity of the gradient we have
$$|\nabla w_\Om(\tilde{x}_0,0)|= \Vert\nabla w_\Om\Vert_\infty$$
The concavity of $w_\Om^{\frac{1}{2}}$ restricted to $\Om\cap\{x_d=\eps\}$ for $\eps\to 0$ implies the concavity of
$$\tilde{x}\in\partial\Om\cap\{x_d=0\}\mapsto \partial_{x_d}w_\Om(\tilde{x},0)^\frac{1}{2},$$
so we have
$$\forall t\in (0,1),\ |\nabla w_{\Om}(t\tilde{x}_0,0)|=\Vert \nabla w_\Om\Vert_\infty.$$
By Lemma \ref{lem_isolatedfailpoint} applied to $x\mapsto w_\Om((\tilde{x}_0,0)+x)$, this implies (by analyticity of $w_\Om$ in $\Om$) that $e\cdot\nabla w_\Om=0$ in $\Om$ for some non-zero vector $e$, which is in contradiction with the fact that $\Om$ is bounded.

To finish, note that the  flat part can not contain another fail point. If this was the case, the concavity of $w_\Om^\frac 12$ would imply that all the segment relying the two fail points consists on fail points, in contradiction with Lemma \ref{lem_isolatedfailpoint}.

This concludes the proof.
\end{proof}
 \begin{remark}\label{m01.1} 
 The statement of Theorem \ref{th_inward_max} gives a qualitative result on the structure of the boundary of a an inward maximiser in a neigbourhood of a fail point: the presence of a flat ball of controlled size centered at the fail point and the absence of any other fail point on the flat part. A slightly different result, of quantitative nature, can also be deduced out of the proof of Theorem \ref{th_inward_max}.
 \end{remark}
 
\begin{proposition}\label{m01} Under the hypothesis of Theorem \ref{th_inward_max}, if $K$ denotes the flat part of the boundary containing a fail point, then
$$\left(\frac{2}{d\om_d}\int_{\mathbb{S}^{d-2}}\frac{1}{\rho_K(\theta)}\d \theta\right)^{-1}\geq \Vert\nabla w_\Om\Vert_\infty.$$
In particular,
$$\mathcal{H}^{d-1}(K)\geq \om_{d-1}\left(K_d \Vert \nabla w_\Om\Vert_\infty\right)^{d-1},$$
where $K_d=\frac{2\Gamma\left(\frac{d}{2}\right)}{\pi^\frac{1}{2}\Gamma\left(\frac{d-1}{2}\right)}$.
\end{proposition}
\begin{proof}
With the notation of the proof of Theorem \ref{th_inward_max}, we have obtained in \eqref{dba11}
$$\frac{\partial h_\vps}{\partial x_d}(0)=-\frac{\Gamma(\frac d2)}{\pi^\frac d2} \int_{\mathbb S^{d-2}} \frac{1}{\rho_{K_\vps} (\theta)}\d\theta.$$
As a consequence, using \eqref{dba10} we get
\begin{equation}\label{d102_bis}
 \frac{\partial w_{\Om_\vps}}{\partial x_d}(\tilde x(\vps), \vps))\ge \frac{\partial w_{\Om}}{\partial x_d}(\tilde x(\vps), \vps))+ w_{\Om} (\tilde x(\vps), \vps))\frac{\Gamma(\frac d2)}{\pi^\frac d2} \int_{\mathbb S^{d-2}} \frac{1}{\rho_{K_\vps} (\theta)}\d\theta.
 \end{equation}
Putting together \eqref{eq_estimateM'(t)} and \eqref{d102},
$$\frac{\partial w_{\Om_\vps}}{\partial x_d}(\tilde x(\vps), \vps))\ge \|\nabla w_\Om\|_\infty -\vps+ w_{\Om} (\tilde x(\vps), \vps))\frac{\Gamma(\frac d2)}{\pi^\frac d2} \int_{\mathbb S^{d-2}} \frac{1}{\rho_{K_\vps} (\theta)}\d\theta.$$
Since by hypothesis $\Om$ is an inward maximiser, we get
$$\|\nabla w_\Om\|_\infty \ge \|\nabla w_{\Om_\vps}\|_\infty \ge \frac{\partial w_{\Om_\vps}}{\partial x_d}(\tilde x(\vps), \vps)),$$
so
\begin{equation}\label{d103_bis}
\vps\ge w_{\Om} (\tilde x(\vps), \vps))\frac{\Gamma(\frac d2)}{\pi^\frac d2} \int_{\mathbb S^{d-2}} \frac{1}{\rho_{K_\vps} (\theta)}\d\theta.
 \end{equation}
Dividing by $\vps$ and passing to the limit, we get
$$1 \ge \|\nabla w_\Om\|_\infty \frac{\Gamma(\frac d2)}{\pi^\frac d2} \limsup_{\vps \to 0^+}  \int_{\mathbb S^{d-2}} \frac{1}{\rho_{K_\vps} (\theta)}\d\theta
$$
proving the first assertion of the Proposition.

In order to get a lower bound for $ \int_{\mathbb S^{d-2}} \frac{1}{\rho_{K_\vps} (\theta)}d\theta$ we rely on Cauchy-Schwarz and Jensen inequalities
$$ \int_{\mathbb S^{d-2}} \frac{1}{\rho_{K_\vps} (\theta)}\d\theta \ge \frac{|\mathbb S^{d-2}|^2}{\int_{\mathbb S^{d-2}} {\rho_{K_\vps} (\theta)}\d\theta }\ge
\frac{|\mathbb S^{d-2}|^\frac{d}{d-1}}{(\int_{\mathbb S^{d-2}} {\rho_{K_\vps}^{d-1} (\theta)}\d\theta)^\frac{1}{d-1} }
= \frac{|\mathbb S^{d-2}|^\frac{d}{d-1}} {\H^{d-1}(K_\vps)^\frac{1}{d-1}}.
$$
Finally
$$\H^{d-1}(K_\vps)^\frac{1}{d-1}\ge \|\nabla w_\Om\|_\infty \frac{\Gamma(\frac d2)}{\pi^\frac d2} |\mathbb S^{d-2}|^\frac{d}{d-1},$$
meaning, after taking $\vps \to 0$, that $\partial \Om$ contains a flat part (which itself contains $0$) with $(d-1)$ Hausdorff measure at least equal to
$$\Big (\|\nabla w_\Om\|_\infty \frac{\Gamma(\frac d2)}{\pi^\frac d2} |\mathbb S^{d-2}|^\frac{d}{d-1} \Big)^{d-1}.$$
Since
$$ \omega_{d-1}^\frac{1}{d-1}K_d = \frac{\Gamma(\frac d2)}{\pi^\frac d2} |\mathbb S^{d-2}|^\frac{d}{d-1} $$
we get the conclusion.

\end{proof}

\EEE

 \section{Proof of Theorem \ref{d20}}\label{sec_areaperimeterconstraint}
 Below,  $G(\Om)$ denotes   either  $|\Om|$ or $ \H^{d-1}(\partial \Om)$. Before the proof, let us briefly describe the rough idea for the symmetry.

\smallskip \noindent {\bf Steiner symmetrisation.} An axial Steiner symmetrisation does not decrease the gradient of the torsion function of $\Om$ at the point where the axis intersects the boundary,   provided the axis is aligned with a normal direction at that point. We shall prove this fact, based on results of Alvino, Trombetti,  Diaz, Lions from \cite[Theorem 1.1 and Remark 3.4]{ATDL96}. If an optimal domain is not Steiner axially symmetric and the constraint is on the perimeter, then the symmetrised domain with respect to a suitably chosen axis has a maximum stress that is at least as large, while having strictly smaller perimeter. Rescaling the symmetrised domain then yields a contradiction with optimality. In the case of a volume constraint, however, the volume is preserved by Steiner symmetrisation, so this argument does not apply directly. We expect that, for a domain that is not axially symmetric, a suitable Steiner symmetrisation should strictly increase the maximum stress. Proving this, however, would require a rigidity result for the maximum stress under Steiner symmetrisation, which is currently not available.

  \begin{proof}[Proof of Theorem \ref{d20}]
 The proof of the existence of a solution follows the direct method of the calculus of variations in any of the three situations, relying on the nondegeneracy of a maximising sequence, the compactness of the Hausdorff metric in balls and the continuity result of Theorem \ref{th_stability_maxstress}.

 \medskip
 \noindent \underline{Proof of 1.} This is a direct consequence Theorem \ref{th_inward_max} and Proposition \ref{d120}.

 \medskip
 \noindent \underline{Proof of 2.}
 Let $\Om$ be a minimiser for either $G(\Om)=\Hd(\partial \Om)$ or $G(\Om)= |\Om|$. Let $x_0$ be a point at the boundary such that $\|\nabla w_\Om\|_\infty= \lim_{x\to x_0} |\nabla w_\Om(x)|$. Without restricting the generality, we assume that $x_0=0$ and $\Om \subset \R^{d-1}\times \R_+$. Then, we know that $0$ is not a corner point and it belongs to a flat part of the boundary. Let us prove that there are no corner points on $\partial \Om$, meaning that $\partial \Om$ is of class $C^1$.

 Assume for contradiction that $y_0 \in \partial \Om$ is a corner point. Then there are two tangent hyperplanes to $\Om$ at $y_0$, say with outer normals $\nu_1$ and $\nu_2$. Let $\nu_0=\frac 12 (\nu_1+\nu_2)$, and
 let $\vps >0$.
 We denote the hyperplane $\H_\vps=\{y : (y-y_0,\nu_0) =-\vps\}$, the half spaces $\H_\vps^-=\{y : (y-y_0,\nu_0) <-\vps\}$, $\H_\vps^+=\{y : (y-y_0,\nu_0) >-\vps\}$  and
 $$\Om_\vps= \Om \cap  \H_\vps^-.$$
Let $K_\vps$ be the convex set $\H_\vps \cap \Om$ and
  denote $a_\vps = \Hd(K_\vps)$. From convexity and comparison with the cone of vertex $y_0$ and basis $K_0$, we get $|\Om \sm \Om_\vps|\ge \frac 1d \vps a_\vps$. At the same time, in view of Grisvard's results, one can use a  barrier argument with an infinite cylinder of two-dimensional basis to get  a constant $C$ depending on the diameter of $\Om$ such that
 $$w_\Om(x) \le C|x-y_0|^{\frac{\pi}{\pi-\alpha} \wedge 2},$$
 where $\alpha \in (0, \pi)$ is the angle between $\nu_1$ and $\nu_2$.

 Throughout the proof, we denote by $C$, $C_1, C_2$ constants which depend on $\Om$, the dimension of the space, the points $x_0, y_0,$ but not on $\vps$. The values of these constants may change from line to line since their effective value is not relevant. At every step, their dependencies will be pointed out.

 \medskip
 We will obtain an estimate for the variation of the gradient $L^\infty$-norms
 $\|\nabla w_\Om\|_\infty-\|\nabla w_{\Om_\vps}\|_\infty$.
 It suffices to get an estimate of $\frac{\partial h_{\vps}}{\partial x_d} (0)$, where
 $$h_\vps:\Om_\vps \to \R, h_\vps =w_\Om-w_{\Om_\vps}.$$
We consider the larger domain denoted $U_\vps\supseteq \Om_\vps$ which is obtained by intersecting $\R^{d-1}\times \R_+$ with $\H_\vps^-$ and $\{-diam (\Om) +\alpha_i (\vps) <x_i<diam (\Om)+\beta_i(\vps)\}$, for $i=1, \dots, d$. Here $\alpha_i, \beta_i$ are close to zero, chosen such that all $U_\vps$ are homothetic to $U_0$, obtained when $\vps=0$ in which case $\alpha_i(0)=\beta_i(0)=0$. Then,
 $$0<\frac{\partial h_{\vps}}{\partial x_d} (0)\le C\vps^{\frac{\pi}{\pi-\alpha} \wedge 2} \frac{\partial H_{\vps}}{\partial x_d} (0),$$
 where
 $$-\Delta H_\vps =0 \mbox{ in } U_\vps, H_\vps= 1_{K_\vps} \mbox{ on } \partial U_\vps.$$
 Denoting by $G_\vps$ the Poisson kernel in $U_\vps$, we get
 $$\frac{\partial H_{\vps}}{\partial x_d} (0)= \int_{K_\vps} \frac{\partial G_\vps}{\partial x_d}  (0,y) \d\sigma_y \le \max_{y\in K_\vps}  \Big |\frac{\partial G_\vps}{\partial x_d} (0,y)\Big | a_\vps.$$
  There exists some $\vps_0,C$ such that for all $ \vps<\vps_0,$
 $$ \forall y \in K_{\vps}, \;  \Big |\frac{\partial G_\vps}{\partial x_d} (0,y)\Big | \le C.$$
 Consequently, for the volume constraint, we have that

 $$\frac{\|\nabla w_{\Om_\vps}\|_\infty}{|\Om_\vps|^\frac 1d}\ge   \frac{\|\nabla w_{\Om}\|_\infty- C  \vps^{\frac{\pi}{\pi-\alpha} \wedge 2}a_\vps }{\left(|\Om|-\frac 1d \vps a_\vps\right)^\frac 1d}.$$
For $\vps>0$ small, the right-hand side is strictly larger than $\frac{\|\nabla w_{\Om}\|_\infty } {|\Om|^\frac 1d}$, contradicting the optimality of $\Om$ for the volume constraint.

 The case of the perimeter is similar. We have
 $$\frac{\|\nabla w_{\Om_\vps}\|_\infty}{\Hd(\partial \Om_\vps)^\frac 1{d-1}}\ge   \frac{\|\nabla w_{\Om}\|_\infty- C_\Om \vps^{\frac{\pi}{\pi-\alpha} \wedge 2}a_\vps }{(\Hd(\partial \Om)+a_\vps -\Hd (\partial \Om \sm \partial \Om_\vps))^\frac 1{d-1}}.$$
The proof can be achieved in a similar way as for the volume constraint, provided the following inequality holds true
 \begin{equation}\label{d50}
 \Hd (\partial \Om \sm \partial \Om_\vps) -a_\vps \ge C a_\vps,
 \end{equation}
for a constant $C$ independent of $\vps$ (possibly depending on $\Om$ and the dimension).

 In order to prove \eqref{d50}, we identify $\H_\vps$ with $\R^{d-1}$ and we assume that  the origin $0$ of this system of coordinates in $\R^{d-1}$ is the centre of the John ellipsoid $E_\vps \sq K_\vps \sq (d-1) E_\vps$ of $K_\vps$, aligned with the axes. We can assume that the semiaxes of $E_\vps$ satisfy
 $$a_{1, \vps} \ge a_{2, \vps} \ge \dots \ge a_{d-1, \vps}.$$
 Then, for two constants $C_1, C_2$ depending on $\Om$ and on the dimension, we have
 $$C_1 \vps \le a_{d-1, \vps} \le C_2 \vps.$$
 Assume $y_0'$ is the projection of $y_0$ onto $\H_\vps$. We see
 the boundary of $\Om$ as a graph over $K_\vps$ associated to the function $f_\vps$. Then, as a consequence of the concavity of $f_\vps$, of the fact that it vanishes at the boundary of $K_\vps $,  and of the inclusion properties of the John ellipsoid, we get
 $$\vps \ge f_\vps(0) \ge \frac{\vps}{2(d-1)}.$$
 Let us denote $g_\vps(x)= f_\eps(0) (1-\rho_{K_\vps} (x))$, where $\rho_{K_\vps}:K_\vps \to \R_+$ is the gauge function of the convex set $K_\vps$.  Then, $g_\vps$ is concave on $K_\vps$, vanishes on its boundary and $g_\vps \le f_\vps$.
 Consequently,
 $$\Hd (\partial \Om \sm \partial \Om_\vps) -a_\vps \ge \int_{K_\vps} \sqrt{1+|\nabla g_\vps|^2} \d x -\int_{K_\vps} 1 \d x= \int_{K_\vps} \frac{|\nabla g_\vps|^2}{\sqrt{1+|\nabla g_\vps|^2} +1} \d x.$$
 Let us introduce the map
 $$T_\vps: \R^{d-1}\to \R^{d-1}, T_\vps (x_1, \dots, x_{d-1}) = (\frac{x_1}{a_{1, \vps}}, \dots, \frac{x_{d-1}}{a_{d-1, \vps}})$$
 and build the blow-up function $u_\vps : T_\vps(K_\vps) \to \R$ defined by
 $$u_\vps(T_\vps(x))= g_\vps(x).$$
 By direct computation we have that
 $$\int_{K_\vps} \frac{|\nabla g_\vps|^2}{\sqrt{1+|\nabla g_\vps|^2} +1} \d x= \int_{T_\vps(K_\vps)} \frac{\ds \sum_{i=1}^{d-1} \frac{1}{a_{i,\vps}^2}\Big(\frac{\partial u_\vps}{\partial x_i}\Big)^2}{\sqrt{1+\ds \sum_{i=1}^{d-1}  \frac{1}{a_{i,\vps}^2}\Big(\frac{\partial u_\vps}{\partial x_i}\Big)^2} +1} \ds \prod_{i=1}^{d-1}  a_{i,\vps}\d x.$$
 Since
 $$B_1 \sq T_\vps(K_\vps) \sq (d-1) B_1,$$
 we get, using a barrier argument at every point, that
 $$\|\nabla u_\vps \|_\infty\le f_\vps(0).$$
 Consequently,
 $$\int_{K_\vps} \frac{|\nabla g_\vps|^2}{\sqrt{1+|\nabla g_\vps|^2} +1} \d x\ge \frac{\Hd(E_\vps)}{\om_{d-1}}
 \int_{T_\vps(K_\vps)} \frac{ \frac{1}{a_{d-1,\vps}^2}\Big(\frac{\partial u_\vps}{\partial x_{d-1}}\Big)^2}{\sqrt{1+\ds \sum_{i=1}^{d-1}  \frac{1}{a_{i,\vps}^2} f_\vps^2(0)} +1}\d x.$$
Moreover,
 $$\frac{ f_\vps^2(0)}{a_{i,\vps}^2}\le \frac{\vps^2}{a_{i,\vps}^2} \le \frac{a_{d-1,\vps}^2}{C_1^2 a_{i,\vps}^2}\le \frac{1}{C_1^2}$$
and
 $$ \frac{1}{a_{d-1,\vps}^2 }\ge \frac{1}{C_2^2 \vps^2} \ge  \frac{1}{C_2^2 f_\vps^2(0) 4 (d-1)^2}.$$
 Denoting by $C$ a constant depending only on $\Om$ and the dimension, we get
  $$\int_{K_\vps} \frac{|\nabla g_\vps|^2}{\sqrt{1+|\nabla g_\vps|^2} +1} dx\ge C \Hd(E_\vps) \int_{T_\vps(K_\vps)} \Big(\frac{\partial U_\vps}{\partial x_{d-1}}\Big)^2 dx,$$
 where $U_\vps(x)= 1- \rho_{T_\vps(K_\vps)}(x)$, by $ \rho_{T_\vps(K_\vps)}$ denoting the gauge function of $T_\vps(K_\vps)$. In particular $\|\nabla U_\vps\|_\infty \le 1$.

 We claim that
$$ \int_{T_\vps(K_\vps)} \Big(\frac{\partial U_\vps}{\partial x_{d-1}}\Big)^2 \d x\ge C>0,$$
where $C$ is a dimensional constant.  Indeed, by contradiction, if there exist convex sets $A_n$, $B_1 \sq A_n\sq (d-1) B_1$, such that
$ \int_{A_n} \Big(\frac{\partial U_{A_n}}{\partial x_{d-1}}\Big)^2 \d x \to 0$, then for  a subsequence converging in Hausdorff metric to $A$, we get
$ \int_{A} \Big(\frac{\partial U_{A}}{\partial x_{d-1}}\Big)^2 \d x = 0$, which is impossible since $A \not = \emptyset$.

This implies \eqref{d50}  since $\Hd (E_\vps) \ge \frac 1{(d-1)^2} a_\vps$.

  \medskip
 \noindent \underline{Proof of 3.}
 We  consider the axial Steiner symmetrisations of $\Om$ around the axis $\R\ov e_d$. Note that $G(\Om^*) \le m$ for the perimeter, while equality occurs for the volume. If this inequality was strict, a homothethy of factor larger than 1 can satisfy the constraint with equality, while the maximum stress is strictly increasing. In such a way we build a new axially symmetric domain $\Om^*$ for which the gradient at $0$ is not smaller than $|\nabla w_{\Om}(0)|$. This last assertion is a consequence of the main result from \cite[Theorem 1.1, assertion (1.12)]{ATDL96} via the characterisation of the boundary gradient from Theorem \ref{th_existence_failpoint}.

Since $\Om$ was a maximiser, this implies that $G(\Om^*) = m$ and that $\|\nabla w_{\Om^*}\|_\infty = c$, so that $\Om^*$ is another maximiser.

 \medskip
 \noindent \underline{Proof of 4. (i)} Assume $\Om$ is optimal for the  perimeter constraint. We know that $\Om^*$ is also optimal and that $\Hd (\partial \Om^*)= \Hd(\partial \Om)$. Indeed, in general, $\Hd (\partial \Om^*)\le \Hd(\partial \Om)$. However, in our case the inequality cannot be strict. Otherwise, take a homothety of $\Om^*$ of ratio larger than $1$ to saturate the perimeter condition. Then the maximum stress would strictly increase contradicting the optimality of $\Om*$. Consequently, $\Hd (\partial \Om^*)= \Hd(\partial \Om)$. Relying on the stability result \cite[Theorem 1.5]{BCF13}, we get that $\Om=\Om^*$. Now, because of the radial symmetry, the only possible flat part in the boundary of $\Om^*$  is perpendicular to $\ov e_d$, so there are at most two of them, among which one is the bottom.

 \medskip
 \noindent \underline{Proof of 4. (ii)} Assume $\Om_0$ and $\Om_1$ are maximisers in $\R^{d-1}\times \R_+$, axially symmetric about $\R \ov e_d$, with $0$ at the boundary   at a point where the maximum of the gradient is attained. The value of both gradients  is assumed to be $c$. As a consequence of  \cite[Theorem 20]{Co05}, for the point $(0', \delta)$ we have that
 $$[w_{(1-t) \Om_0+t\Om_1} (0', \delta)]^\frac 12\ge (1-t)[w_{  \Om_0 } (0', \delta)]^\frac 12+ t [w_{ \Om_1} (0', \delta)]^\frac 12$$
 which gives, after dividing by $\delta^\frac 12$ and passage to the limit,
 $$\frac{\partial w_{(1-t) \Om_0+t\Om_1} }{\partial x_d}(0', \delta)\ge c.$$
This proves (ii) since the perimeter (in $\R^2$)  of $(1-t) \Om_0+t\Om_1$ is not more than $m$.
 \end{proof}

  \begin{remark}\label{d65}\rm
 Assume the value of $m$ in Theorem \ref{d20} is respectively $|B_1|$, $\H^{d-1}(\partial B_1)$. Then  the radius of the flat disk centred at the fail point for the corresponding maximisation problem is at least
 $(\pi d)^{-1}$. Indeed, we know that  the optimal set $\Om$ is not a worse competitor than the ball of radius $1$, so that  $\|\nabla w_\Om\|_\infty \ge \frac 1d$. Theorem \ref{d20} implies the assertion.
\end{remark}

\begin{remark} \rm In case of the volume constraint, we can not prove that all solutions have to be Steiner symmetric. To do so, we would need a rigidity result in the argument based on \cite[Theorem 1]{ATDL96}. This is  somehow related to a maximum principle of Hopf type for a degenerate elliptic operator. There are intuitive arguments in favour of the validity of such a result. However a proof requires deep analysis.
  \end{remark}

 \section{Further remarks: various constraints, quantitative estimates and open problems}\label{sec_furtherremarks}
 In this section we discuss the maximisation of the maximal stress under different constraints. We prove existence and several geometric characterisations for maximisers of the maximum stress under diameter, circumradius, torsional rigidity and principal Dirichlet eigenvalue constraints. All solutions for these problems are inward maximisers for the maximum stress, as a consequence of Proposition \ref{d120}.

 \noindent{\bf The diameter constraint.} Let us assume that $G(\Om)= diam(\Om)$.

 \begin{prop}\label{d20.d}

 Let $m>0$. The following problem
 \begin{equation}\label{d43.d}
 \max\{\|\nabla w_\Om\|_\infty : \Om \subset \R^d, \Om \mbox{ open and convex}, diam(\Om)=m\}
 \end{equation}
 has at least one solution say  $\Om$. Assume, up to a rigid movement, that $\Om \subset \R^{d-1} \times \R_+$ and  $0 \in \partial \Om$ is a fail point.
Then there exists a body of constant width $\mathcal C,$ axially symmetric about the axis $\R\ov e_d,$ such that  $\Om=\mathcal C \cap [\R^{d-1}\times \R_+]$ and  $\mathcal C \cap [\R^{d-1}\times \R_-]$ is a spherical cap.
Moreover, if $\Om_0$ and $\Om_1$ are two solutions as above, then for every $t \in (0,1)$ their Minkowski sum
 $(1-t) \Om_0+t\Om_1$ is also a solution.
  \end{prop}
\begin{proof}
Existence of a solution follows using the direct method of the calculus of variations.
Assume $\Om \subset \R^{d-1} \times \R_+$ is such a solution, and that $0$ is a fail point. Relying on \cite[Corollary 3.2.7]{Sc14}, we know that $\Om$ is contained in a convex set $\mathcal C$ of constant width equal to the diameter. Then the set $\mathcal C \cap [\R^{d-1} \times \R_+]$ satisfies the diameter constraint and contains $\Om$. Since $\Om$ is optimal
 $\Om= \mathcal C\cap [\R^{d-1} \times \R_+]$ and, we get that $\Om^* \sq  \mathcal C^*\cap [\R^{d-1} \times \R_+]$. On the other hand, the set $\mathcal C^*$ has a diameter not larger than $m$ and it is itself contained in a convex set of constant width equal to $m$, say $\mathcal  K$, which can be chosen to be axially symmetric. Clearly, $\Om^*$ has to coincide with $\mathcal K\cap [\R^{d-1} \times \R_+]$ (otherwise it would not be optimal) hence it coincides with $ \mathcal C^*\cap [\R^{d-1} \times \R_+]$.

 Let us denote by $V_0$ the $(d-1)$-dimensional ball $\partial \Om^* \cap [\R^{d-1}\times \{0\}]$ and observe that the point $X^*$ of maximal height in $\Om^*$ has to be at distance $m$ from any point of $\partial_{\R^{d-1}\times \{0\}} V_0$. Indeed, if this would not be the case, then the width constraint in $\Om^*$ would not saturate at $X^*$, so we could find a larger domain than $\Om^*$ without exceeding the diameter constraint. This implies that there exists a point $X$ at the same height in $\Om$. Taking now the ball centred at $X$ of radius $m$, this ball should contain $\Om$, while its intersection with $\R^{d-1}\times \{0\}$ is a ball congruent with $V_0$. Consequently,  $\partial \Om \cap [\R^{d-1}\times \{0\}] \equiv V_0$ and the body $C$ of constant width $m$ has a boundary in $\R^{d-1}\times \R_-$ which is a spherical cap on $\partial B(X, m)$. This is invariant under Steiner symmetrisation, so $C^*=\mathcal K$. Consequently, we are in the equality case for the mean width  under axial Steiner symmetrisation. But, there exists a specific sequence of hyperplanes containing the line $\R \ov e_d$ such that an iterative sequence of Steiner symmetrisations with respect to these hyperplanes are converging to the axial one \cite{Vo16}.  Since at each step the mean width can not increase and equality occurs only if the set is already symmetric with respect to the hyperplane,  we get that  $\Om$ is congruent to $\Om^*$.
 \end{proof}

\smallskip
 \noindent{\bf The circumradius constraint.} Below we consider the case $G(\Om)= R(\Om)$, the radius of the smallest ball containing $\Om$.
\begin{prop}\label{the7}
  Let $d\ge 2$. The variational problem
\begin{equation}\label{e27p}
c:=\sup\{\|\nabla w_{\Om}\|_{\infty}:\Om\, \textup{ non-empty, open, bounded, convex},\,\Om\subset \R^d,\, R(\Om)=m\},
\end{equation}
has a chorded ball with radius $m$ as a maximiser. Moreover,
\begin{itemize}
\item [(i)]
$$
c\ge \frac md\Big(1+\frac{K_d^2}{4}\Big)^{1/2},
$$
where $K_d$ is the constant of Proposition \ref{m01},
$$
K_d  =\frac{2\Gamma(\frac d2)}{\pi^\frac 12 \Gamma(\frac{d-1}2) }.
$$
In particular for $d=2$
$$c>0.5247m \, .$$
The radius $r_m$ of the chord of the maximiser satisfies
$$
r_m\ge \frac{m}{d}K_d\Big(1+\frac{K_d^2}{4}\Big)^{1/2}.
$$
\item [(ii)] If $d=2$, then
\begin{equation*}
c< 0.6404\,m.
\end{equation*}
\end{itemize}

\end{prop}
\medskip
\noindent \textit{Proof of}(i)  Let $(\Om_n)$ be a maximising sequence, $\Om_n \sq B_m$. For each $n\in \N$ there exists a point $P_n\in \partial \Om_n$ such that
$\|\nabla w_{\Om_n}\|_{\infty}=|\nabla w_{\Om_n}|(P_n).$ Let $\H_n$ be the hyperplane tangent to $\partial \Om_n$ at $P_n$. Let $\H_n^+$ be the half-space with boundary $\H_n$ which contains $\Om_n.$
 Since  $|\nabla w_{\Om}|(p_n)\le w_{\H_n^+\cap B_{R(\Om_n)}}$, we have, by monotonicity,
$|\nabla w_{\Om_n}|(P_n)\le |\nabla w_{\H_n^+\cap B_{m}}|(P_n)$.
This means that $(\H_n^+\cap B_{m})_n$ is also a maximising sequence which, up to extraction of a subsequence, converges to a chorded ball in the Hausdorff metric. The continuity result, Theorem \ref{th_stability_maxstress}, concludes the proof of the existence part.

Without loss of generality, we put $m=1$ and $\Om=D$, the unit disc centred at $(0',1)$. We have that
\begin{equation}\label{g18}
w(x',x_d)=(1-|x'|^2-(x_d-1)^2)/(2d).
\end{equation}
Hence
 \begin{equation*}
\|\nabla w_{D}\|_{\infty}=|\nabla w_D|(0',0)=\frac1d.
\end{equation*}
Let $D_{\delta}=\{x\in D:x_d>\delta\}$.
From Theorem \ref{the5} with $\Om=D$ and \eqref{g18}, we get that
 \begin{equation}\label{g22}
\|\nabla w_{D_{\delta}}\|_{\infty}\ge \frac{1}{d}(1-\delta)+K_d\frac{1}{2d}(2\delta-\delta^2)^{1/2}.
\end{equation}
Taking the supremum over all $\delta\in (0,1)$ yields that
the right-hand side of \eqref{g22} is maximal for
$$
\delta_d=1-\frac{2}{(4+K_d^2)^{1/2}}.
$$
Hence
\begin{equation*}
c\ge \|\nabla w_{D_{\delta_d}}\|_{\infty}\ge \frac1d\Big(1+\frac{K_d^2}{4}\Big)^{1/2}.
\end{equation*}
The lower bound for $r_m$ follows from Theorem \ref{the5}(iii) below.
\smallskip

\noindent \textit{Proof of} (ii) The proof relies on a probabilistic representation of the torsion function. The torsion function plays a key role in probability theory. Let $(B(s),\,s\ge 0, \Pa_x,\,x\in \R^d)$
 be Brownian motion with generator $\Delta$, and let
\begin{equation*}
T(\Om)=\inf\{s\ge 0: B(s)\not\in \Om\}
\end{equation*}
be the lifetime of Brownian motion in $\Om$. Then
\begin{equation*}
\Pa_x[T_{\Om}>t]=v_{\Om}(x;t),
\end{equation*}
where $v_{\Om}$ is the unique solution of
\begin{equation}\label{e7}
\Delta v=\frac{\partial v}{\partial t},\,x\in \Om,\,v\mid_{\partial\Om}=0,\,\,t>0,
\end{equation}
with initial condition $v(x;0)=1,\,x\in\Om.$

The expected value of the random variable $T(\Om)$ is denoted and given by
\begin{equation}\label{e8}
\mathbb{E}_x[T(\Om)]=\int_{\R^+}\d t\,\Pa_x[T(\Om)>t]=\int_{\R^+}\d t\,v_{\Om}(x;t).
\end{equation}
The right-hand side of \eqref{e8} is a representation for the torsion function $w_{\Om}$.

\smallskip

We let
\begin{equation*}
 B(\varepsilon)=\Big\{x\in\R^2: \Big(x_1-\frac{\varepsilon}{2}\Big)^2+x_2^2<1,\,x_1<1\Big\}
\end{equation*}
be a chorded unit disc,
and let $C(\varepsilon)$ be the largest disc in $B(\varepsilon)$ given by
\begin{equation*}
C(\varepsilon)=\Big\{x\in\R^2: |x|<1 -\frac{\varepsilon}{2}\Big\}.
\end{equation*}
Note that the tangent planes of $B(\varepsilon)$ and $C(\varepsilon)$ at the point $(1-\frac{\varepsilon}{2})$ coincide. Let $x(a)=(1-\frac{\varepsilon}{2}-a,0)$, where $a>0$ is small.
We have that
\begin{equation}\label{g30}
\Pa_{x(a)}[T_{B(\varepsilon)}>t]=\Pa_{x(a)}[T_{C(\varepsilon)}>t]+\Pa_{x(a)}[T_{C(\varepsilon)}\le t<T_{B(\varepsilon)}].
\end{equation}
Let $V(x(a),1-\frac{\varepsilon}{2},\tau,\theta)$ be the joint probability density of Brownian motion exiting at time $\tau$ in time interval $\d\tau$, and at $((1-\frac{\varepsilon}{2})\cos \theta,(1-\frac{\varepsilon}{2})\sin \theta)$ in the angle interval $\d\theta$. By the strong Markov property at $\tau$,
 $$
\Pa_{x(a)}[T_{C(\varepsilon)} \le t<T_{B(\varepsilon)}]\hskip 10.5cm
$$
$$
\le \int_{(-\pi,\pi)}\d\theta\int_{(0,t)}\d\tau V(x(a),1-\frac{\varepsilon}{2},\tau,\theta)\Pa_{((1-\frac{\varepsilon}{2})\cos \theta,(1-\frac{\varepsilon}{2})\sin \theta)}[t-\tau\le T_{B(\varepsilon)}].
$$
Integrating  \eqref{g30} with respect to $t$ over $[0,\infty)$ gives that
$$
w_{B(\varepsilon)}(x(a))\le w_{C(\varepsilon)}(x(a))\hskip 9cm$$
$$\hskip 2cm+\int_{(-\pi,\pi)}\d\theta\int_{[0,\infty)}\d \tau V(x(a),1-\frac{\varepsilon}{2},\tau,\theta)w_{B(\varepsilon)}((1-\frac{\varepsilon}{2})\cos \theta,(1-\frac{\varepsilon}{2})\sin \theta).
$$Since $\int_{[0,\infty)}\d \tau V(x(a),1-\frac{\varepsilon}{2},\tau,\theta)$ is the density of the harmonic measure of a disc with radius $1-\frac{\varepsilon}{2}$  at the point $x(a)$ in $\d\theta$ \cite[Chapter 1]{GM} we obtain that

$$
w_{B(\varepsilon)}(x(a))\le w_{C(\varepsilon)}(x(a))\nonumber\hskip 11cm$$
$$+\frac{1}{2\pi}\int_{(-\pi,\pi)}\d\theta \frac{(1-\frac{\varepsilon}{2})^2-|x(a)|^2}{(1-\frac{\varepsilon}{2})^2-2(1-\frac{\varepsilon}{2})|x(a)|\cos \theta+|x(a)|^2}w_{B(\varepsilon)}((1-\frac{\varepsilon}{2})\cos \theta,(1-\frac{\varepsilon}{2})\sin \theta).
$$
 Multiplying  both sides by $a^{-1}$ and taking the limit $a\downarrow 0$ gives that (recall that  $w_{B(\varepsilon)}(x(0))=w_{C(\varepsilon)}(x(0))=0$)
\begin{align}\label{g34}
|\nabla w_{B(\varepsilon)}|(1-\frac{\varepsilon}{2},0)&\le |\nabla w_{C(\varepsilon)}|(1-\frac{\varepsilon}{2},0)+\frac{1}{2\pi}\int_{(-\pi,\pi)}\d\theta\frac{w_{B(\varepsilon)}((1-\frac{\varepsilon}{2})\cos \theta,(1-\frac{\varepsilon}{2})\sin \theta)}{(1-\frac{\varepsilon}{2})(1-\cos \theta)}\nonumber\\&
=\frac12\Big(1-\frac{\varepsilon}{2}\Big)+\frac{1}{2\pi}\int_{(-\pi,\pi)}\d\theta\frac{w_{B(\varepsilon)}((1-\frac{\varepsilon}{2})\cos \theta,(1-\frac{\varepsilon}{2})\sin \theta)}{(1-\frac{\varepsilon}{2})(1-\cos \theta)}.
\end{align}
The coordinates of the straight line segment in $\partial B(\varepsilon)$ are denoted by
\begin{equation}\label{g35}
S_1=(1-\frac{\varepsilon}{2},(2\varepsilon-\varepsilon^2)^{1/2}),\,\,S_2=(1-\frac{\varepsilon}{2},-(2\varepsilon-\varepsilon^2)^{1/2})
\end{equation}
respectively. We put $O=(0,0)$ and $2\theta_{\varepsilon}=\angle S_1 OS_2$ so that
\begin{equation}\label{g36}
\sin \theta_\varepsilon=\Big(\frac{4\varepsilon}{3\varepsilon +2}\Big)^{1/2},\,\,\cos\theta_\varepsilon=\Big(\frac{2-\varepsilon}{2+3\varepsilon}\Big)^{1/2}.
\end{equation}
We put
\begin{equation}\label{g37}
g(\varepsilon)=|\nabla w_{B(\varepsilon)}|(1-\frac{\varepsilon}{2},0).
\end{equation}
For any point on the straight line segment  $P_{\theta}=(1-\frac{\varepsilon}{2})(1,\tan\theta)$,\,$\theta\le \theta_\varepsilon,$ we have that
\begin{equation}\label{g38}
|\nabla w_{B(\varepsilon)}|(P_{\theta})\le g(\varepsilon).
\end{equation}
The distance from a point $((1-\frac{\varepsilon}{2})\cos \theta),(1-\frac{\varepsilon}{2})\sin \theta)$ on $\partial C(\varepsilon),
$ to the line segment $[S_1,S_2]$ is given by $(1-\frac{\varepsilon}{2})(1-\cos \theta)$.
Hence
\begin{equation}\label{g39}
w_{B(\varepsilon)}((1-\frac{\varepsilon}{2})\cos \theta),(1-\frac{\varepsilon}{2})\sin \theta)\le (1-\frac{\varepsilon}{2})(1-\cos \theta)g(\varepsilon),\,0<\theta\le \theta_{\varepsilon}.
\end{equation}
For points $\theta\in [\theta_{\varepsilon},\pi]\cup[-\theta_{\varepsilon},-\pi],$ we use monotonicity to obtain
\begin{align}\label{g40}
w_{B(\varepsilon)}((1-\frac{\varepsilon}{2})\cos \theta),(1-\frac{\varepsilon}{2})\sin \theta)&\le w_{B_{0}}((1-\frac{\varepsilon}{2})\cos \theta),(1-\frac{\varepsilon}{2})\sin\theta)\nonumber\\&
=\frac14 \varepsilon (1-\frac{\varepsilon}{2})(1+\cos \theta).
\end{align}
By \eqref{g34}--\eqref{g40},
\begin{align}\label{g41}
g(\varepsilon)&\le \frac12\Big(1-\frac{\varepsilon}{2}\Big)+\frac{g(\varepsilon)}{\pi}\int_{(0,\theta_{\varepsilon})}\d \theta+\frac{\varepsilon}{4\pi}\int_{[\theta_{\varepsilon},\pi]}\d \theta\frac{1+\cos \theta}{1-\cos \theta}\nonumber\\&
=\frac12\Big(1-\varepsilon\Big)+\frac{1}{\pi}\Big(g(\varepsilon)+\frac{\varepsilon}{4}\Big)\theta_{\varepsilon}+\frac{\varepsilon}{2\pi}\frac{1+\cos \theta_{\varepsilon}}{\sin \theta_{\varepsilon}},
\end{align}
where we have used
$$\int \d \theta\frac{1+\cos\theta}{1-\cos\theta}=-\theta-\frac{2(1+\cos \theta)}{\sin\theta}. $$
Rearranging \eqref{g41} and using \eqref{g36} yields
 $$
g(\varepsilon)\Big(1-\frac{1}{\pi}\arcsin\Big(\Big(\frac{4\varepsilon}{3\varepsilon+2}\Big)^{1/2}\Big)\Big)\hskip 6cm$$
$$\hskip 2cm\le \frac12(1-\varepsilon)+\frac{\varepsilon}{4\pi}\arcsin\Big(\Big(\frac{4\varepsilon}{3\varepsilon+2}\Big)^{1/2}\Big)+ \frac{\varepsilon^{1/2}}{4\pi}
\Big((2+3\varepsilon)^{1/2}+(2-\varepsilon)^{1/2}\Big).
$$
The coefficient of $g(\varepsilon)$ in the left-hand side above  is strictly positive since  $0\le \varepsilon\le 1$. It follows that, $\forall  0\le \varepsilon\le 1,$
$$
g(\varepsilon)\le \frac{\frac12(1-\varepsilon)+\frac{\varepsilon}{4\pi}\arcsin\Big(\Big(\frac{4\varepsilon}{3\varepsilon+2}\Big)^{1/2}\Big)+ \frac{\varepsilon^{1/2}}{4\pi}
\big((2+3\varepsilon)^{1/2}+(2-\varepsilon)^{1/2}\big)}{1-\frac{1}{\pi}\arcsin\Big(\Big(\frac{4\varepsilon}{3\varepsilon+2}\Big)^{1/2}\Big)}.
 $$
The right-hand side achieves its maximum value for some $\varepsilon\in(0.3240,0.3241)$. That maximum is bounded from above by $0.6404$ (computations performed with Wolfram Alpha \cite{WA}).

\hfill$\square$

\begin{remark} \rm
 Let  $m=G(B_1)$, and let $c$ be the value of the maximum in \eqref{d43} with $G$ either volume, perimeter or diameter and $\Om$ a maximiser of the corresponding problem \eqref{d43}.
Then
$$
c\ge \frac1d\Big(1+\frac{K_d^2}{4}\Big)^{1/2}.
$$
Indeed, the value of $c$ is not smaller than the one corresponding to the maximiser of \eqref{e27p} with $m=1$.
\end{remark}

\hfill$\square$

\smallskip
\noindent {\bf Torsion constraint.}
\begin{theorem}\label{torsion}
The variational expression
\begin{equation}\label{c1}
\sup\{\|\nabla w_{\Om}\|_{\infty}: \Om\, \textup{open, convex in}\, \R^d, \|w_{\Om}\|_1=1\}
\end{equation}
has a maximiser. If $\Om^*$ is a solution, then it is an inward maximiser for the maximum stress, and
\begin{equation}\label{c2}
diam(\Om^*)\le \frac{4\omega_d^{(d+1)/(d+2)}d^{d+4+\frac{1}{d+2}}(d+1)^{(d+1)/2}(d+2)^{2+\frac{1}{d+2}}}{\omega_{d-1}}.
\end{equation}
\end{theorem}

\begin{proof}
To show that the supremum in \eqref{c1} is finite it suffices to show that the scaling invariant quantity
\begin{equation*}
H(\Om):=\frac{\|\nabla w_{\Om}\|_{\infty}}{\big(\|w_{\Om}\|_1\big)^{1/(d+2)}}
\end{equation*}
is bounded from above. First observe that for a convex set $\Om\in\R^d$ with $\|w_{\Om}\|_1=1$ we have that $\textup{diam}(\Om)<\infty$. However, the sole requirements open, convex and $\|w_{\Om}\|_1=1$
are insufficient to obtain an upper bound for $\textup{diam}(\Om)$ as, in general, there is no lower bound for the inradius $r(\Om):=\sup\{d_{\Om}(x):x\in\Om\}$.
By the second inequality in \eqref{d00} we have for the bounded set $\Om$,
\begin{equation}\label{c4}
H(\Om)\le \frac{(2\|w_{\Om}\|_{\infty})^{1/2}}{\big(\|w_{\Om}\|_1\big)^{1/(d+2)}}.
\end{equation}
Let $d_{\Om}:\Om\rightarrow [0,\infty)$ denote the distance to the boundary function,
$$
 d_{\Om}(x)=\inf\{|x-y|: y\in \partial\Om\}.
$$
We have $B(x;d_{\Om}(x))\subset \Om$. By monotonicity of the torsion function, we have for $\Om_1\subset \Om,$
 $$
w_{\Om}(x)\ge w_{\Om_1}(x),\,\,x\in \Om_1.
 $$
Then $\Om$ contains a ball $B_{r_{\Om}/2}$ with radius $r(\Om)/2$.
By monotonicity of the torsional rigidity
\begin{equation}\label{c7}
\|w_{\Om}\|_1 \ge \|w_{B_{r_{\Om}/2}}\|_1= \frac{\omega_d}{d(d+2)}\Big(\frac{r(\Om)}{2}\Big)^{d+2}.
\end{equation}
By Hersch's inequality \cite{He60}, we have for an open, convex $\Om\subset \R^d$ with finite inradius,
\begin{equation}\label{c8}
\lambda(\Om)\ge \frac{\pi^2}{4r(\Om)^2}.
\end{equation}
By \eqref{c7} and \eqref{c8},
\begin{equation}\label{c9}
\|w_{\Om}\|_1\ge \frac{\omega_d}{d(d+2)}\Big(\frac{\pi}{4}\Big)^{d+2}\lambda(\Om)^{-(d+2)/2}.
\end{equation}
By \cite[Theorem 1]{vdBC},
\begin{equation}\label{c10}
\|w_{\Om}\|_{\infty}\le (4+3d\log 2)\lambda(\Om)^{-1}.
\end{equation}
By \eqref{c9}--\eqref{c10}
\begin{equation*}
\|w_{\Om}\|_{\infty}\le\frac{16(4+3d\log 2)}{\pi^2}\Big(\frac{d(d+2)}{\omega_d}\|w_{\Om}\|_1\Big)^{2/(d+2)}.
\end{equation*}
By \eqref{c4} and \eqref{c10},
\begin{equation}\label{c12}
H(\Om)\le \pi^{-1}(32(4+3d\log 2))^{1/2}\Big(\frac{d(d+2)}{\omega_d}\Big)^{1/(d+2)}.
\end{equation}
Put
\begin{equation}\label{c13}
c^*:=\sup\{\|\nabla w_{\Om}\|_{\infty}: \Om\, \textup{open, convex in}\, \R^d, \|w_{\Om}\|_1=1\}.
\end{equation}
By \eqref{c12}, $c^*<\infty$. If the supremum in \eqref{c13} is achieved for the ball, then there is nothing to prove. If not, then let $(\Om_n)$ be a maximising sequence.
Without loss of generality, we may assume that $(H(\Om_n))$ is increasing , $\lim_{n\rightarrow \infty}H(\Om_n)=c^*$ and
\begin{equation}\label{c14}
c^*> H(B_{\rho_d}),
\end{equation}
where $B_{\rho_d}$ is the ball with radius $\rho_d$ such that $\|w_{B_{\rho_d}}\|_1=1$.

Consider an element $\Om$ of the sequence $(\Om_n)$. By the second inequality in \eqref{d00} we have that $c^*\le \frac12 h(\Om)$. Hence  \eqref{c14} gives that
\begin{equation*}
\frac12 h(\Om)> \frac{\rho_d}{d}
\end{equation*}
where $h(\Om)$ is the width of $\Om$. By Steinhagen's inequality \cite{St},
\begin{equation}\label{c16}
h(\Om)\le \begin{cases}
2d^{1/2}r(\Om), & \text{if $d$ is odd }  \\
\frac{2(d+1)}{(d+2)^{1/2}}r(\Om), & \text{if $d$ is even}.
\end{cases}
\end{equation}
To simplify \eqref{c16} we just write $h(\Om)\le 2(d+1)^{1/2}r(\Om)$.
This gives
\begin{equation}\label{c17}
r(\Om)\ge \frac{\rho_d}{d(d+1)^{1/2}}.
\end{equation}
It has been shown, for example in \cite{PiGa14}, that
\begin{equation*}
\|w_{\Om}\|_1\ge \frac{|\Om|r(\Om)^2}{d(d+2)}.
\end{equation*}
This, together with \eqref{c17}, gives that
\begin{equation*}
|\Om|\le d^3(d+1)(d+2)\frac{\|w_{\Om}\|_1}{\rho_d^2}.
\end{equation*}
Since $\Om$ contains a ball with radius $r(\Om)$ and a line segment from the centre of that ball with length at least $\frac12\textup{diam}(\Om)$, $\Om$ contains a cone with a $(d-1)$-dimensional ball with radius $r(\Om)$, and height $\frac12\textup{diam}(\Om).$ So, we find that
\begin{equation}\label{c20}
diam(\Om)\le \frac{2d|\Om|}{\omega_{d-1}r(\Om)^{d-1}}.
\end{equation}
By \eqref{c17}--\eqref{c20},
\begin{equation*}
diam(\Om)\le \frac{2d^{d+2}(d+1)^{(d+1)/2}(d+2)\|w_{\Om}\|_1}{\omega_{d-1}\rho_d^{d+1}}.
\end{equation*}
Recall that $\|w_{\Om}\|_1=1$, and that $\|w_{B_{\rho_d}}\|_1=\frac{\omega_d \rho_d^{d+2}}{d(d+2)}=1$.
This gives that
\begin{align}\label{c22}
diam(\Om)&\le \frac{2\omega_d d^{d+3}(d+1)^{(d+1)/2}(d+2)^2}{\omega_{d-1}}\rho_d\nonumber \\&
\le \frac{2\omega_d d^{d+4}(d+1)^{(d+2)/2}(d+2)^2}{\omega_{d-1}}r(\Om),
\end{align}
where we used \eqref{c17} in the last inequality.
By \eqref{c7} and \eqref{c22},
\begin{equation}\label{c23}
diam(\Om)\le \frac{4\omega_d^{(d+1)/(d+2)}d^{d+4+\frac{1}{d+2}}(d+1)^{(d+1)/2}(d+2)^{2+\frac{1}{d+2}}}{\omega_{d-1}}.
\end{equation}
There exists a cuboid $C_d$ with sidelength depending upon $d$ only which contains suitable translates of $\Om_1,\Om_2,...$, denoted by $\Om_1,\Om_2,...$.
There exists a subsequence of $(\Om_n)$ again denoted by $(\Om_n)$  which converges in the complementary Hausdorff metric to an element say $\Om^*$.
By continuity of $\|\nabla w_{\Om}\|_{\infty}$ and $\|w_{\Om}\|_1$,  $\Om^*$ is a maximiser of \eqref{c1}. By  continuity of the diameter and \eqref{c23}, implies \eqref{c2}.
Note that the ball $B_{\rho_d}$ also satisfies \eqref{c2}.
\end{proof}

\smallskip
\noindent {\bf Eigenvalue constraint.}

For an open, non-empty, bounded and convex set in $\R^d,$ we put
\begin{equation}\label{d0}
J(\Om)=\|\nabla w_{\Om}\|_{\infty}\lambda(\Om)^{1/2}.
\end{equation}

\smallskip

\begin{theorem}\label{the2}
\begin{itemize}
\item[\textup{(i)}]The variational expression
\begin{equation}\label{d1}
\sup\Big\{J(E): E\subset \R^2 \textup{ non-empty, open, ellipse}\, \Big\}
\end{equation}
has a maximiser. If $E^*$ is a maximising ellipse, then
\begin{equation*}
J(E^*)>\frac{\pi}{2}.
\end{equation*}

\item[\textup{(ii)}] The variational expression
\begin{equation}\label{d3}
\sup\Big\{\|\nabla w_{\Om}\|_{\infty}: \Om\subset \R^2, \textup{non-empty, open, bounded and convex,}\, \lambda(\Om)=1 \Big\}
\end{equation}
has a  maximiser. If $\Om^*$ is a solution, then it is an inward maximiser for the maximum stress, and
\begin{equation}\label{d4}
diam(\Om^*)\le \max\Big\{\textup{diam}(E^{**}),3j_0\Big(1+2\Big(\Big(\frac{2J(E^*)}{\pi}\Big)^{2/3}-1\Big)^{-3}\Big),
\end{equation}
where $j_0$ is the first positive zero of the Bessel function $J_0$, and $E^{**}$ is a maximiser of \eqref{d1} with $\lambda(E^{**})=1$.
\end{itemize}
 \end{theorem}

\noindent{\it Proof.}(i)
 Consider the ellipse $E_{1,n}$ with axes $1$ and $n,\,n\ge 1.$ It was shown in \cite[formula (1.3)]{BF10} that
$$
\lb(E_{1,n})=\frac{\pi^2}{4}+\frac{\pi}{2n}+o(n^{-1}), \, n\rightarrow \infty.
$$
Since
\begin{equation*}
w_{E_{1,n}}(x_1,x_2)=\frac{1-x_1^2-\frac{x_2^2}{n^2}}{2(1+n^{-2})},
\end{equation*}
we find that
\begin{equation*}
\|\nabla w_{E_{1,n}}\|_{\infty}=\frac{n^2}{1+n^2}.
\end{equation*}
Hence
\begin{equation*}
\|\nabla w_{E_{1,n}}\|_{\infty}\lb (E_{1,n})^{1/2}=\frac{\pi}{2}+\frac{1}{2n}+o(n^{-1}),  \, n\rightarrow \infty.
\end{equation*}
This is for all $n$ sufficiently large, strictly larger than the value $\frac{\pi}{2}.$
There exists a maximiser by continuity and boundedness of $n \rightarrow J(E_{1,n})$ on $[1,\infty)$.

\smallskip

(ii) We first show that the supremum in \eqref{d3} is finite. It suffices to prove that the scaling invariant quantity $J(\Om)$ is bounded. By Payne's inequality in \eqref{d00}
\begin{equation}\label{d5}
\|\nabla w_{\Om}\|_{\infty}\le \frac{h(\Om)}{2}.
\end{equation}
Since $\Om$ contains a disc with radius $r(\Om),$ we have by monotonicity of the Dirichlet spectrum that
\begin{equation}\label{d6}
\lambda(\Om)\le \frac{j_0^2}{r(\Om)^2}.
\end{equation}
By \eqref{c16} for $d=2$, and \eqref{d5}--\eqref{d6},
\begin{equation}\label{d6_bis}
J(\Om)\le \frac{3j_0}{2}.
\end{equation}
Denote the supremum in \eqref{d3} by $c^*$. Let $E^{**}$ be a maximiser of \eqref{d1} with $\lambda(E^{**})=1$. There are two cases. Either $E^{**}$ is a maximiser of \eqref{d3} or
$c^*>J(E^{**})$. In the first case  $E^{**}$ trivially satisfies \eqref{d4}. In the second case we let $(\Om_n)$ be a maximising sequence. Without loss of generality, we may assume that
$n\rightarrow J(\Om_n)$ is increasing and that $J(\Om_n)>J(E^*)$. Consider an arbitrary element $\Om$ in this sequence.

\smallskip

For a non-empty, open, bounded and convex set in $\R^2$, it can be shown that there exist $z_1,z_2\in \partial \Om$ which support two parallel lines defining $h(\Om)=|z_1-z_2|$, and such that the closed line segment $[z_1,z_2]$ is perpendicular to these lines.  This follows from the proof of \cite[Theorem 1.5]{vdB2}. The two supporting lines are given by $x_2=\pm \frac12h(\Om)$. It was shown in \cite[p.597]{vdB2} that if $\Om'=\{(x_1,0): (x_1,x_2)\in \Om \}$  denotes the projection of $\Om$ onto the line $x_2=0$, and if $\Lambda(\Om')$ denotes the bottom of the spectrum of the $1$-dimensional Dirichlet Laplacian acting in $L^2(\Om')$, then
\begin{equation}\label{h31}
\lambda(\Om)\le \frac{\pi^2}{h(\Om)^2}\left(1+c(\Om)\right)^3,
\end{equation}
where
\begin{equation}\label{h32}
c(\Om):=\Big(\frac{4h(\Om)^2\Lambda(\Om')}{\pi^2}\Big)^{1/3}.
\end{equation}
By \eqref{d0}, \eqref{d5}, \eqref{h31} and \eqref{h32},
$$J(\Om)\le \frac{\pi}{2}\Big(1+\Big(\frac{4h(\Om)^2\Lambda(\Om')}{\pi^2}\Big)^{1/3}\Big)^{3/2}.$$
Since $J(\Om_n)>J(E^*),$ we obtain that
\begin{equation*}
J(E^*)\le \frac{\pi}{2}\Big(1+\Big(\frac{4h(\Om)^2\Lambda(\Om')}{\pi^2}\Big)^{1/3}\Big)^{3/2}.
\end{equation*}
Since $\Lambda(\Om')=\frac{\pi^2}{|\Om'|^2}$, where $|\Om'|$ is the length of the line segment of the projection on $x_2=0,$ we find that
\begin{equation}\label{d8}
|\Om'|\le 2h(\Om)\Big(\Big(\frac{2J(E^*)}{\pi}\Big)^{2/3}-1\Big)^{-3}.
\end{equation}
Let $d_{1}$ and $d_{2}$ be points which realise the diameter of $\Om$, and let $\tilde{d}_{1}, \tilde{d}_{2}$ be their projections onto $x_2=0$. Since $\Omega\subset \{-\frac12 h(\Om) \leq x_2\leq \frac12 h(\Om)\}$ we have that,
$|d_{1}-\tilde{d}_{1}|\le \frac12h(\Om),\,|d_{2}-\tilde{d}_{2}|\le \frac12h(\Om)$. Hence by \eqref{d8}
\begin{align}\label{h40}
\textup{diam} (\Om)&=|d_{1}-d_{2}|\nonumber\\&\le |\tilde{d}_{1}-\tilde{d}_{2}|+h(\Om)\le |\Om'|+h(\Om)\nonumber\\&
\le h(\Om)\Big(1+2\Big(\Big(\frac{2J(E^*)}{\pi}\Big)^{2/3}-1\Big)^{-3}\Big).
\end{align}
By \eqref{c16} for $d=2$ and \eqref{d6_bis}, we have that
$h(\Om)\le 3r(\Om)\le 3j_0\lambda(\Om)^{-1/2}$. This gives with \eqref{h40}
\begin{equation*}
\textup{diam}(\Om)
\le 3j_0\lambda(\Om)^{-1/2}\Big(1+2\Big(\Big(\frac{2J(E^*)}{\pi}\Big)^{2/3}-1\Big)^{-3}\Big).
\end{equation*}
For each element of the minimising sequence we have $\lambda(\Om)=1$. This shows that the diameters of the elements $\Om_n$ are uniformly bounded.
There exists a subsequence of translates of $\Om_n$ which converges in the complementary Hausdorff metric to say $\Om^*$. By continuity, $\Om^*$ is a minimiser and its diameter satisfies \eqref{d4}.

\hfill$\square$
\begin{remark}\rm
The equilateral triangle is not a maximiser since, by \cite{P},
$$\|\nabla w_{T_1}\|_{\infty}\lb(T_1)^{1/2}=\frac{\pi}{2}<J(E^*).$$
 \end{remark}

In Theorem \ref{the5}  we give, under hypothesis H below, a probabilistic alternative  to a particular case of  inequality \eqref{d102_bis} contained in  the proof of  Theorem \ref{th_inward_max} and Proposition \ref{m01}\EEE.  We consider a non empty, open, bounded and convex set $\Om$. Without loss of generality we may assume that $\Om\subset \R^{d-1}\times \R^+$, and that $(0',0)\in\partial \Om $ is a failpoint of $\Om$. For $\delta>0$ we let
$w_{\delta}$ be the torsion function for the set  $\Om_{\delta}=\Om\,\cap \{x_d>\delta\},$ and  $w_{\Om}=w.$  Then
$$
\nabla w_{\delta}(0',\delta)=\Big(0',\frac{\partial w_{\delta}}{\partial x_d}(0',\delta)\Big).
$$
We make the following hypothesis
$$\max_{x\in \Om\cap\{ x_d=\delta\}}w_{\Om}(x)=w_{\Om}(0',\delta),\,\delta>0.\,\,\textup{(H)}$$

\smallskip

\begin{theorem}\label{the5}

If $\Om$ satisfies (H), then
$$
\frac{\partial w_{\delta}}{\partial x_d}(0',\delta)\ge \frac{\partial w}{\partial x_d}(0',\delta) +w(0',\delta)\frac{\Gamma(d/2)}{\pi^{d/2}}\int_{{  \{x_d=\delta\}\setminus\Om}}\frac{\d x'}{|x'|^d},\,\delta>0.
$$

\end{theorem}

\begin{proof}
 Let $x(a)=(0',\delta+a)$ be a point in $\Om_{\delta},$ where $a$ is small and positive.
We have that
\begin{equation}\label{g4}
\Pa_{x(a)}[T_{\Om}>t]=\Pa_{x(a)}[T_{\Om_{\delta}}>t]+\Pa_{x(a)}[T_{\Om_{\delta}}\le t < T_{\Om}].
\end{equation}
We let $B=(B',B_d)$ where $B_d$ is a one-dimensional Brownian motion taking values in $\R$, and $B'$ is an independent $(d-1)$-dimensional Brownian motion taking values in $\R^{d-1}$.
Furthermore we let
\begin{equation*}
\tau_{\delta}=\inf\{s\ge 0: B_d=\delta\}.
\end{equation*}
We verify the following inclusion of events
\begin{equation}\label{g4b}
\{T_{\Om_{\delta}}\le t < T_{\Om}\}\subseteq\{\tau_{\delta}\le t<T_{\Om}\}.
\end{equation}
By \eqref{g4b}, and the strong Markov property at the stopping time $\tau_{\delta}$ we have that
\begin{align}\label{g4c}
\Pa_{x(a)}[T_{\Om_{\delta}}\le t<T_{\Om}]&\le \Pa[\tau_{\delta}\le t<T_{\Om}]\nonumber \\&
=\E_{x(a)}\Big\{\int_{(0,t)}1_{\{\tau_{\delta}\in \d \tau\}}\Pa_{(B'(\tau),\delta)}[t-\tau<T_{\Om}]\Big\}\nonumber\\&
=\int_{(0,t)}\E_{x(a)}\Big[ 1_{\{\tau_{\delta}\in \d \tau\}}\Pa_{(B'(\tau),\delta)}[t-\tau<T_{\Om}]\Big].
\end{align}
where ${\bf 1}$ denotes the indicator function.

By solving the heat equation \eqref{e7} for the half-line $H=(0,\infty),$
we find that
\begin{equation}\label{g5}
\Pa_{c}[T_H<t]=\frac{2}{\pi^{1/2}}\int_{(c/(4t)^{1/2},\infty)}\d\theta\, e^{-\theta^2},\,c>0.
\end{equation}
Let $\tau_c=\inf\{s:B_d(s)=c\}$.
The random variable $\tau_c$ has a density $\rho(\tau;c)$ given by differentiating \eqref{g5} with respect to $t$,
\begin{equation}\label{g6}
\rho(\tau;c)=\frac{c}{(4\pi \tau^3)^{1/2}}e^{-c^2/(4\tau)}{\bf 1}_{\{\tau\ge 0\}}.
\end{equation}

A further standard fact is that for any bounded measurable function $f:\R^{d-1}\rightarrow \R$, any $x'\in \R^{d-1}$, and any $\tau>0$,
\begin{equation}\label{g7}
\E_{x'}[f(B(\tau))]=\frac{1}{(4\pi \tau)^{(d-1)/2}}\int_ {\R^{d-1}}\d y'\,f(y') e^{-|x'-y'|^2/(4\tau)}.
\end{equation}

Since the distance of $x(a)$ to the line $x_d=\delta$ equals $a$ we find, with $c=a$ in \eqref{g5} that
\begin{equation}\label{g7a}
\E_{x(a)}\Big[ 1_{\{\tau_{\delta}\in \d \tau\}}\Big]=\rho(\tau;a)\d \tau.
\end{equation}
By \eqref{g7} applied to $f(x')=\Pa_{(x',\delta)}[t-\tau<T_{\Om}]$ we obtain that
\begin{equation}\label{g7b}
\E_{x(a)}\Big[\Pa_{(B'(\tau),\delta)}[t-\tau<T_{\Om}]\Big]=\frac{1}{(4\pi \tau)^{(d-1)/2}}\int_{\R^{d-1}} \d x'\, e^{-|x'|^2/(4\tau)}\Pa_{(x',\delta)}[t-\tau<T_{\Om}].
\end{equation}
By \eqref{g4c}, \eqref{g6}, \eqref{g7a}, \eqref{g7b} and the independence of $B_d$ and $B'$
\begin{align}\label{g8}
\Pa_{x(a)}[T_{\Om_{\delta}}\le t<T_{\Om}]&\le \int_{(0,t)}\d \tau \frac{a}{(4\pi)^{d/2}\tau^{(d+2)/2}}\int_{\R^{d-1}}e^{-(a^2+|x'|^2)/(4\tau)}{\bf 1}_{\{\tau\ge 0\}}\Pa_{(x',\delta)}[t-\tau<T_{\Om}]\Big]\nonumber\\&
=\int_{(0,\infty)}\d \tau \frac{a}{(4\pi)^{d/2}\tau^{(d+2)/2}}\int_{\R^{d-1}}e^{-(a^2+|x'|^2)/(4\tau)}{\bf 1}_{\{0\le \tau\le t\}}\Pa_{(x',\delta)}[t-\tau<T_{\Om}].
\end{align}
Integrating \eqref{g8} with respect to $t$ over $[0,\infty)$ yields by Fubini's theorem
\begin{align}\label{g9}
\int_{[0,\infty)} \d t&\Pa_{x(a)}[T_{\Om_{\delta}}\le t<T_{\Om}]\nonumber\\&\le
\int_{(0,\infty)}\d \tau \frac{a}{(4\pi)^{d/2}\tau^{(d+2)/2}}\int_{\R^{d-1}}\d x'e^{-(a^2+|x'|^2)/(4\tau)}\int_{[0,\infty)}\d t{\bf 1}_{\{0\le \tau\le t\}}\Pa_{(x',\delta)}[t-\tau<T_{\Om}]\nonumber\\&
=\int_{(0,\infty)}\d \tau \frac{a}{(4\pi)^{d/2}\tau^{(d+2)/2}}\int_{\R^{d-1}}\d x'e^{-(a^2+|x'|^2)/(4\tau)}\int_{[0,\infty)}\d t\Pa_{(x',\delta)}[t<T_{\Om}]\nonumber\\&
=\int_{(0,\infty)}\d \tau \frac{a}{(4\pi)^{d/2}\tau^{(d+2)/2}}\int_{\R^{d-1}}\d x'e^{-(a^2+|x'|^2)/(4\tau)}w(x',\delta)\nonumber\\&
=\int_{(0,\infty)}\d \tau \frac{a}{(4\pi)^{d/2}\tau^{(d+2)/2}}\int_{\{x':(x',\delta)\in\Om\}}\d x'e^{-(a^2+|x'|^2)/(4\tau)}w(x',\delta)\nonumber\\&
\le \int_{(0,\infty)}\d \tau \frac{a}{(4\pi)^{d/2}\tau^{(d+2)/2}}\int_{\{x':(x',\delta)\in\Om\}}\d x'e^{-(a^2+|x'|^2)/(4\tau)}w(0',\delta)\nonumber\\&
=\frac{a\Gamma(d/2)}{\pi^{d/2}}\int_{\Om\cap\{x_d=\delta\}}\frac{\d x'}{(a^2+|x'|^2)^{d/2}}w(0',\delta),
\end{align}
where we have used (H) in the last inequality. By \eqref{g9} and \cite[3.194.3]{GR},

\begin{align}\label{g11}
\int_{(0,\infty)}\d t\,\Pa_{x(a)}&[T_{\Om_{\delta}}\le t<T_{\Om}]\le
\frac{a\Gamma(d/2)}{\pi^{d/2}}\int_{\R^{d-1}}\frac{\d x'}{(a^2+|x'|^2)^{d/2}}w(0',\delta)\nonumber\\&
-\frac{a\Gamma(d/2)}{\pi^{d/2}}\int_{{  \{x_d=\delta\}\setminus\Om}}\frac{\d x'}{(a^2+|x'|^2)^{d/2}}w(0',\delta)\nonumber\\&
=w(0',\delta)-\frac{a\Gamma(d/2)}{\pi^{d/2}}w(0',\delta)\int_{{ \{x_d=\delta\}\setminus\Om}}\frac{\d x'}{(a^2+|x'|^2)^{d/2}}.
\end{align}
By \eqref{g4}, \eqref{g4c}, \eqref{g11}, and recalling that $w_{\delta}(0',\delta)=0$, we have
\begin{align}\label{g12}
w_{\delta}&(0',\delta+a)-w_{\delta}(0',\delta)\nonumber\\&\ge w(0',\delta+a)-w(0',\delta)+\frac{a\Gamma(d/2)}{\pi^{d/2}}\int_{{  \{x_d=\delta\}\setminus\Om}}\frac{\d x'}{(a^2+|x'|^2)^{d/2}}w_{\Om}(0',\delta).
\end{align}
Multiplying both sides of \eqref{g12} by $\frac1a$, and taking the limit $a\downarrow 0$ yields
\begin{align*}
\frac{\partial w_{\delta}}{\partial x_d}(\delta,0')- \frac{\partial w}{\partial x_d}(\delta,0')&\ge \lim_{a\downarrow 0}\frac{\Gamma(d/2)}{\pi^{d/2}}\int_{{  \{x_d=\delta\}\setminus\Om}}\frac{\d x'}{(a^2+|x'|^2)^{d/2}}\nonumber\\&
=w(0',\delta)\frac{\Gamma(d/2)}{\pi^{d/2}}\int_{{  \{x_d=\delta\}\setminus\Om}}\frac{\d x'}{|x'|^d}.
\end{align*}

\end{proof}

\smallskip
 \medskip

\noindent {\bf Acknowledgements.} The authors would like to thank the Isaac Newton Institute for Mathematical Sciences, Cambridge, for support and hospitality during the programme Geometric Spectral Theory and applications, where work on this paper was undertaken. This work was supported by EPSRC grant EP/Z000580/1. DB was partially supported by a grant from the Simons Foundation and by the ANR project STOIQUES financed by the French Agence Nationale de la Recherche (ANR-24-CE40-2216). MvdB wishes   to thank Katie Gittins   and Ilaria Lucardesi for helpful discussions. DB wishes to thank Roberto Ognibene for a useful suggestion to shorten the proof of  Lemma \ref{d04}. We wish to thank Cristina Trombetti for helpful discussions concerning Steiner symmetrisation.

\end{document}